\documentclass[12pt]{article}
\usepackage{amsmath,amsfonts,amssymb,mathtools}
\usepackage{amsthm}
\usepackage{amsmath}
\usepackage{authblk}
\usepackage[hmargin=2.5cm,vmargin=3cm,centering]{geometry}
\usepackage{hyperref}
\hypersetup{
    colorlinks,
    citecolor=black,
    filecolor=magenta,
    linkcolor=black,
    urlcolor=black,
    }
\usepackage[titletoc]{appendix}
\usepackage{color}
 \definecolor{red}{rgb}{1,0,0}
\newcommand{\xddots}{%
  \raise 4pt \hbox {.}
  \mkern 6mu
  \raise 1pt \hbox {.}
  \mkern 6mu
  \raise -2pt \hbox {.}
}

\theoremstyle{break}
\newtheorem{thm}{Theorem}[section] 
\newtheorem{RQ}{Remark}[thm]

\newtheorem{prop}[thm]{Proposition}

\newtheorem{lem}[thm]{Lemma}
\numberwithin{equation}{section}

\DeclareMathOperator \ch {ch}
\DeclareMathOperator \sh {sh}
\DeclareMathOperator \tnh {th}

\title{\large\bf{Quaternionic structure and analysis \\
of some
    Kramers-Fokker-Planck
 operators}}

\author[1]{Mona~Ben~Said \thanks{bensaid@univ-paris13.fr}}
\author[2]{Francis~Nier \thanks{nier@math.univ-paris13.fr}}
\author[3]{Joe~Viola \thanks{ Joseph.Viola@univ-nantes.fr}}
\affil[1,2]{Laboratoire Analyse, G{\'e}om{\'e}trie et Applications\\
   Universit{\'e} Paris 13\\
   99 Avenue Jean Baptiste Cl{\'e}ment \\93430 Villetaneuse, France
 }
\affil[3]{
   Laboratoire de Math{\'e}matiques Jean Leray\\
   Universit{\'e} de Nantes\\
   2, rue de la Houssini{\`e}re\\
   BP 92208, 44322 Nantes Cedex 3, France
 }

\begin{document}

\maketitle
\vspace{-1cm}
\begin{abstract}
The present article is concerned with global subelliptic
estimates for Kramers-Fokker-Planck operators
with polynomials of degree less than or equal to two. 
The constants appearing in those
  estimates are accurately formulated
 in terms of the coefficients, especially when those are
  large.
\end{abstract}

\noindent\textbf{MSC-2010:} 35Q84, 35H20, 35P05, 47A10, 14P10, 20G20 

\tableofcontents

\section{Introduction and main results}

In this work, we consider the {Kramers}-Fokker-Planck operator given by

\begin{align}
K_V=p.\partial_q-\partial_qV(q).\partial_p+\frac{1}{2}(-\Delta_p+p^2),\quad (q,p)\in
  \mathbb{R}^{2d}
\,,
\label{def1}
\end{align}
where $q$ denotes the space variable, $p$ denotes the velocity
variable, $x.y=\sum\limits_{j=1}^{d}x_{j}y_{j}$\,,
  $x^{2}=\sum\limits_{j=1}^{d}x_{j}^{2}$ and  the
potential $V
(q)=\sum\limits_{|\alpha|\leq 2}V_{\alpha}q^{\alpha}$ is a real-valued polynomial
function on $\mathbb{R}^{d}$ with $d^{\circ}V=2$\,.
After making an orthogonal change of variables one may assume that its Hessian matrix is
 $$\mathrm{Hess}\; V=\begin{pmatrix}
\nu_{1} &  0  & \ldots & 0\\
0  &  \nu_{2} & \ldots & 0\\
\vdots & \vdots & \ddots & \vdots\\
0  &   0       &\ldots & \nu_{d}
\end{pmatrix}{~.}
$$
The constant term $V_{0}$ does not appear in $K_{V}$ and can be set to $0$ and we distinguish two cases:

$\bullet$ If $\mathrm{Hess}\; V$ is non-degenerate,  a   translation in $q$
reduces the problem to
\begin{align}
V(q)=\sum_{i=1}^d\frac{\nu_i}{2}q^2_i{\,.}
\label{eq1}
\end{align}

$\bullet$ If $\mathrm{Hess}\; V$ is  degenerate, a good choice of  orthonormal basis gives:

\begin{align}
V(q)=\lambda_1q_1+\sum_{i=2}^d\frac{\nu_i}{2}q^2_i\,,
\label{eq2}
\end{align}
where $\lambda_1$ is invariantly defined as $\min\limits_{q\in\mathbb{R}^{d}}\;|\nabla\;V(q)|\geq
  0$\,.\\

As established in \cite{HeNi} (see Proposition 5.5, page 44), the non-selfadjoint operator $K_V$ is 
maximal accretive when endowed with the
domain 
$D(K_V)=\left\{u\in L^2(\mathbb{R}^{2d}),\;\;\; K_Vu\in
  L^2(\mathbb{R}^{2d})\right\}.$
The question about the compactness of the resolvent
combined with subelliptic estimates is intimately related with the
return to the equilibrium or exponential decay estimates. As pointed
out in \cite{HerNi} and \cite{HeNi}, the analysis of $K_{V}$ is also
strongly related to the one of the Witten Laplacian
$\Delta_{V}^{(0)}=-\Delta_{q}+\left|\nabla V(q)\right|^{2}-\Delta V(q)$ for which
maximal hypoelliptic techniques developed by Helffer and Nourrigat in
\cite{HeNo} provide accurate criteria for general polynomial
potentials $V(q)$\,. 

Within this maximal hypoelliptic analysis of
$\Delta_{V}^{(0)}$ there is a recurrent interplay between qualitative
estimates and quantitative estimates in terms of the size of the
coefficients  of the polynomial $V(q)$\,. The general idea is that the
study of the operator $\Delta_{V}^{(0)}$ as $q\to \infty$ when $V$ is
a degree $r$ polynomial, is reduced to a quantitative version of
subelliptic estimates for $\Delta_{\tau \widetilde{V}}^{(0)}$\,, where
$\widetilde{V}$ belongs to some family of polynomials related to $V$ with
degree less than $r$\,, and $\tau$ is a large
parameter. 

``Quantitative estimates'' means that we consider
subelliptic estimates with a good and optimal control of the constant with respect
to the parameter $\tau$\,. Remember also that the compactness of the
resolvent  on
$\Delta_{V}^{(0)}$ obtained by maximal hypoelliptic techniques\,,
relies on the fact that no polynomial $\widetilde{V}$ of the family
associated with $V$ admits a local minimum. It shows in particular
that the compactness of the resolvent of $\Delta_{+V}^{(0)}$ and
$\Delta_{-V}^{(0)}$ differ and the first non trivial example comes
with the potential $\pm V(q_{1},q_{2} )=\pm q_{1}^{2}q_{2}^{2}$ in
$\mathbb{R}^{2}$\,. For the Kramers-Fokker-Planck operator $K_{V}$\,, no
sufficient condition until the recent work by W.X.~Li \cite{Li2}
exhibited such a different behavior. 

We hope to develop the same strategy for the non self-adjoint
operator $K_{V}$ as for the Witten Laplacian $\Delta_{V}^{(0)}$\,,
namely try to get the optimal subelliptic estimates for some class of
polynomial functions $V(q)$\,, by making use of quantitative estimates
for some lower degree polynomials. The case $d^{\circ}V\leq 2$ for
which the Weyl symbol of $K_{V}$ is a polynomial of degree $\leq 2$ in
the variable $(q,p,\xi_{q},\xi_{p})$ allows a lot of exact analytic
calcultations and was already deeply studied in
\cite{Hor}\cite{Sjo}\cite{HiPr}\cite{Vio}\cite{Vio2}\cite{AlVi}. Nevertheless
exploiting those exact analytic expressions for the semigroup kernel or
symbol (Mehler's type formulas) or for the spectrum does not solve
completely the question of optimal quantitative subelliptic estimates
for the non self-adjoint operator $K_{V}$\,. The
semiclassical regime which can be handled quite accurately via
symbolic calculus gives results after rescaling
essentially when the transport part
$p.\partial_{q}-\partial_{q}V(q).\partial_{p}$ is small compared to
the diffusive--friction part $\frac{-\Delta_{p}+p ^{2}}{2}$\,. 

Actually, we are mainly interested in the other regime where the
Hamiltonian dynamics is stronger than the diffusive and friction
part. The difficulty then appears clearly, because understanding the
operator $K_{V}$ requires the understanding of the Hamiltonian
dynamics associated with
$p.\partial_{q}-\partial_{q}V(q).\partial_{p}$ which, for a general
polynomial $V$ exhibits a rich variety of phenomena, and which, for a
polynomial of degree $\leq 2$, already contains the three types of dynamics:
a) elliptic (bounded trajectories when $V$ is a positive definite
quadratic form); b)
hyperbolic (trajectories escaping exponentially quickly in time  to
infinity when $V$ is a negative definite quadratic form); and c)
parabolic (trajectories escaping polynomially quickly in time to
infinity when $V$ is linear).

At a more fundamental level, understanding the operator $K_{V}$ when the transport
term is dominant also proceeds in the same direction as Bismut's program: in \cite{Bis1}, Bismut introduced his
hypoelliptic Laplacian in order to interpolate Morse
theory (in the
high diffusion-friction regime via the Witten Laplacian) and the
topology of loop spaces (dominant transport term). The difficult part
with a dominant transport term
was understood only for the geodesic flow on symmetric spaces making
use of the specific algebraic structure in \cite{Bis2}.

With this respect our simpler case also requires a better
understanding of the underlying algebra, and it appeared that after
using the general FBI-techniques  the Kramers-Fokker-Planck
evolution with quadratic potentials, even in dimension $d=1$\,, 
 is reduced to some linear
dynamics on $\mathbb{C}^{4}$ which are easily computed after
elucidating some quaternionic
structure. In this specific case, this also completes the unfruitful
attempts in \cite{HeNi}, Section~9.1, to exhibit some useful nilpotent Lie algebra
structure for Kramers-Fokker-Planck operators. Actually, quaternions
and Pauli matrices are related to the $\mathfrak{su}(2)$ Lie algebra, so
the Lie algebra structure decomposition useful to the analysis 
of Kramers-Fokker-Planck operators with polynomial potentials is
certainly not nilpotent.

Denoting 
\[
	O_p=\frac{1}{2}(D^2_p+p^2)
\]
and
\[
X_V=p.\partial_q-\partial_qV(q).\partial_p~,
\] 
we can rewrite the
Kramers-Fokker-Planck operator $K_V$ defined in (\ref{eq1}) as
$K_V=X_V+O_p$. \\
In this work, we are mainly based on recent publications by Hitrik,
Pravda-Starov, Viola, and Aleman \cite{AlVi},
\cite{Vio2}, and \cite{HPV2} which deal with operators having
polynomial symbols of degree less than or equal to two. 

\textbf{Notations:}  
\[
	\begin{aligned}
	\text{Tr}_{+} & = \sum_{\nu_{i>0}} \nu_{i}\,,
	\\ \text{Tr}_{-} &=-\sum_{\nu_{i}\leq 0}\nu_{i}\,,
\\ A &= \max \{(1+\text{Tr}_{+})^{2/3}, 1+\text{Tr}_{-}\}
\,	
\\	B &= \max\{\lambda_1^{4/3}, \frac{1+\text{Tr}_{-}}{\log(2+\text{Tr}_{-})^2}\},
	\end{aligned}
\]

The main goal of this work is the following subelliptic estimates.

\begin{thm}\label{thm1.1}

Let $V(q)$ be a potential as in (\ref{eq1}) or (\ref{eq2}). Then there
exists a constant $c>0$ that does not depend on $V$ such that 
the subelliptic estimate with a remainder term
\begin{align}
\|K_Vu\|^2_{L^2(\mathbb{R}^{2d})}+A\|u\|^2_{L^2(\mathbb{R}^{2d})}\ge {c} \Big(\|O_pu\|^2_{L^2(\mathbb{R}^{2d})}&+\|X_Vu\|^2_{L^2(\mathbb{R}^{2d})}\nonumber\\&+\|\langle\partial_q V(q)\rangle^{2/3}u\|^2_{L^2(\mathbb{R}^{2d})}+\|\langle D_q\rangle^{2/3}u\|_{L^2(\mathbb{R}^{2d})}\Big)\label{eq4}
\end{align}
holds for all $u\in D(K_V)$.
\end{thm}
\begin{thm}\label{thm1.2}

Let $V(q)$ as in (\ref{eq1}) or (\ref{eq2}). Then there is  a constant
$c>0$ independent of the polynomial $V$ so that 
{the subelliptic estimate without a remainder}
\begin{align*}
\|K_Vu\|^2_{L^2(\mathbb{R}^{2d})}\ge \frac{c}{1+\frac{A}{B}}\Big(\|O_pu\|^2_{L^2(\mathbb{R}^{2d})}&+\|X_Vu\|^2_{L^2(\mathbb{R}^{2d})}\nonumber\\&+\|\langle\partial_q V(q)\rangle^{2/3}u\|^2_{L^2(\mathbb{R}^{2d})}+\|\langle D_q\rangle^{2/3}u\|_{L^2(\mathbb{R}^{2d})}\Big)
\end{align*}
{holds}
for all $u\in D(K_V)$, under the condition $\text{Tr}_{-}+\lambda_1\not=0$.
\end{thm}

\begin{RQ}
In view of the comparison with compactness criteria for Witten
Laplacians with polynomial potentials (see  Theorem 10.16 in \cite{HeNi}), note that the 
 condition $\text{Tr}_{-}+\lambda_1\not=0$ imposed in Theorem~\ref{thm1.2} is equivalent to the fact that the potential $V$ does not have a local minimum.  
\end{RQ}

The two previous Theorems are both consequences of the following result.
\begin{prop}\label{prop1.1}

There exists a constant $c>0$ such that
\begin{align*}
 \sum_{i=1}^d
  \| |D_{q_i}| e^{-t(K_V+\sqrt{A})}\|_{{\cal L}(L^{2}(\mathbb{R}^{2d}))}+
 \||\partial_{q_i}V(q_i)|e^{-t(K_V+\sqrt{A})}\|_{\mathcal{L}(L^2(\mathbb{R}^{2d}))}\le\frac{c}{t^{\frac{3}{2}}}
\end{align*}
 for all $t>0$.

Moreover, if $\text{Tr}_{-}+\lambda_1\not=0$,

\begin{align*}
{\|K_{V}^{-1}\|_{{\cal L}(L^{2}(\mathbb{R}^{2d}))}\leq}\int_0^{+\infty}\|e^{-tK_V}\|_{\mathcal{L}(L^2(\mathbb{R}^{2d}))}dt\le
  \frac{c}{\sqrt{B}}
\end{align*}
\end{prop}

\subsection*{Acknowledgements} Joe Viola is grateful for the support of the R\'egion Pays de la Loire through the project EONE (Evolution des Op\'erateurs Non-Elliptiques).

\section{Reduction to a one-dimensional problem}

Interpolation results of Lunardi (see   Remark ~5.11, Theorem 5.12 and Corollary~5.13 in \cite{Lun}) show
 that the first inequality of Proposition~\ref{prop1.1} combined with the fact that 
\begin{align*}
|\mathbb{R}\text{e}\langle [O_p,X_V]u,u\rangle| \le C_{\epsilon}(\|\;|D_q|^{\frac{2}{3}}u\|^2+\|\;|\partial_qV(q)|^{\frac{2}{3}}u\|^2)+\epsilon\|O_pu\|^2
\end{align*}
for all $u\in D(K_V)$ (where $\epsilon>0$ is small enough), implies the
subelliptic estimates given in  Theorem~\ref{thm1.1}.
Theorem~\ref{thm1.2} is then a consequence of Theorem~\ref{thm1.1} and
the second inequality of Proposition~\ref{prop1.1}.\\
Details are given below.
\begin{proof}[Proof of Proposition~\ref{prop1.1}]
Since this result  is expressed in terms of the
semigroup, it can be studied by a separation of variables for a
potential of the form~(\ref{eq1}) or (\ref{eq2}). Actually
$e^{-tK_{V}}$ is a commutative product of contraction semigroups,
and it suffices to write
$$
\sum_{i=1}^d \|M_{i}e^{-t(K_V+\sqrt{A})}\|_{\mathcal{L}(L^2(\mathbb{R}^{2d}))}
\leq 
\sum_{i=1}^{d}\|M_{i}e^{-t(K_{V(q_{i})}+\alpha_{i})}\|_{\mathcal{L}(L^2(\mathbb{R}^{2}))}
\quad M_{i}=|D_{q_{i}}|~\text{or}~M_{i}=|\partial_{q_{i}} V(q)|\,,
$$
where $V(q_{i})$ denotes the one-dimensional potential in the $q_{i}$
variable, with   $V(q_{1})=\frac{\nu_{1}q_{1}^{2}}{2}$ or
$V(q_{1})=\lambda_{1}q_{1}$\,,
$V(q_{i})=\frac{\nu_{i}q_{i}^{2}}{2}$ for $i\geq 2$\,, 
$\alpha_{i}=|\nu_{i}|^{1/2}$ if $\nu_{i}<0$\,,
$\alpha_{i}=\nu_{i}^{1/3}$ if $\nu_{i}>0$ and $\alpha_{i}=0$ if
$\partial_{q_{i}}^{2}V=0$\,. The second estimate of
Proposition~\ref{prop1.1} is even simpler. Hence
Proposition~\ref{prop1.1} will be the result of a careful analysis of
the three one-dimensional potentials  $V(q)=\pm \frac{\nu
  q^{2}}{2}$\,, $\nu>0$\,, and
$V(q)=\lambda_{1}q$\,, $\lambda_{1}\in \mathbb{R}$\,, developed in the
next sections.
\end{proof}
\begin{proof}[Proof of Theorem~\ref{thm1.1}]
In this proof we use nearly the same notations as in \cite{Lun} (Remark ~5.11, Theorem 5.12 and Corollary~5.13). Set
\[
	\begin{aligned}
	T(t)&=e^{-t(\sqrt{A}+K_{V})}~,
	\\ L^2&= L^2(\mathbb{R}^{2d})~,
	\\ E&= \lbrace u\in L^2(\mathbb{R}^{2d}), qu,\partial_qu\in L^2(\mathbb{R}^{2d})\rbrace
	\end{aligned}
\]
where $E$ is equipped with the norm 
\begin{align*}
\|u\|^2_E = \sum_{i=1}^d\||D_{q_i}|u\|_{L^2(\mathbb{R}^{2d})}^2+\||\partial_{q_{i}}V(q_{i})|u\|_{L^2(\mathbb{R}^{2d})}^{2}+\|u\|^2_{L^2(\mathbb{R}^{2d})}~.
\end{align*}
Applying Lemma~\ref{lem3.3} and Proposition~\ref{prop3.1}, we obtain by separation of variables \begin{align*}
\|T(t)\|_{\mathcal{L}(L^2,E)}\le \frac{c}{t^{\frac{3}{2}}}\;\;\;\text{ for all}\;\;\; t>0~{.}
\end{align*}
If $m=3$ and $\beta=\frac{1}{2}$, then by Theorem 5.12 in \cite{Lun}, one has the following embedding of  real interpolation spaces \begin{align}\Big(L^2,D\Big((\sqrt{A}+K_{V})^3\Big)\Big)_{\frac{\theta}{2},p}\subset \Big(L^2,E\Big)_{\theta,p}\label{s}\end{align} for all $\theta \in (0,1),\;p\in[1,+\infty].$
In particular for  $\theta=\frac{2}{3}$, 
\begin{align}
[L^2,E]_{\frac{2}{3}}=(L^2,E)_{ \frac{2}{3},2}={\lbrace u\in L^2,\;|D_{q_i}|^{\frac{2}{3}}u\in L^2,\;|\partial_{q_i}V(q_i)|^{\frac{2}{3}}u\in L^2\;\;\text{for all} \;1\le i\le d\rbrace~,}\label{i}
\end{align}
where the complex interpolation space $[L^2,E]_{\frac{2}{3}}$ is equipped with the norm
\begin{align*}
\|u\|_{[L^2,E]_{\frac{2}{3}}}={\sum_{i=1}^{d}\Big(\|\;|D_{q_i}|^{\frac{2}{3}}u\|_{L^2(\mathbb{R}^{2d})}^2+\|\;|\partial_{q_i}V(q_i)|^{\frac{2}{3}}u\|^2_{L^2(\mathbb{R}^{2d})}\Big)+\|u\|^2_{L^2(\mathbb{R}^{2d})}~.}
\end{align*}
Moreover in view of  Remark ~5.11 and Corollary 5.13 in \cite{Lun},\begin{align}\Big(L^2,D\Big((\sqrt{A}+K_{V})^3\Big)\Big)_{\frac{1}{3},2}=D(\sqrt{A}+K_{V})\label{w}
\end{align}   (since $L^2$ is a Hilbert space and $(\sqrt{A}~+~K_{V})$ is a maximal accretive operator).
Thus taking into account (\ref{s}), (\ref{i}) and (\ref{w}) \begin{align*}
D(\sqrt{A}+K_{V})\subset {\lbrace u\in L^2,\;|D_{q_i}|^{\frac{2}{3}}u\in L^2,\;|\partial_{q_i}V(q_i)|^{\frac{2}{3}}u\in L^2\;\;\text{for all} \;1\le i\le d\rbrace~{.}}
\end{align*}
Hence there exists a constant $c>0$ such that
\begin{align}
{\sum_{i=1}^{d}\Big(\|\;|D_{q_i}|^{\frac{2}{3}}u\|^2+\|\;|\partial_{q_i}V(q_i)|^{\frac{2}{3}}u\|^2_{L^2}\Big)}\le c\|(\sqrt{A}+K_{V})u\|^2_{L^2}\label{eq5}
\end{align}
holds for all $u\in D(K_V)$.

Write for $u\in D(K_V)$, \begin{align}
\|(\sqrt{A}+K_{V})u\|^2_{L^2}=\|(\sqrt{A}+O_p)u\|^2_{L^2}+\|X_Vu\|^2_{L^2}+2\mathbb{R}\text{e}\langle [O_p,X_V]u,u\rangle{~,}
\label{eq6}
\end{align}
so
\begin{align}
|2\mathbb{R}\text{e}\langle [O_p,X_V]u,u\rangle|&\le\sum_{i=1}^d\Big|\mathbb{R}\text{e}\langle
                              u,\Big(D_{p_i}D_{q_i}+p_i\partial_{q_i}V(q)\Big)u\rangle\Big|\nonumber\\&\le
                                                                                                            \sum_{i=1}^d|\mathbb{R}\text{e}\langle
                                                                                                            u,(D_{p_i}D_{q_i})u\rangle|
                                                                                                            +|\mathbb{R}\text{e}\langle
                                                                                                            u,p_i\partial_{q_i}V(q)u\rangle|\nonumber\\&\le
                                                                                                                                                           \sum_{i=1}^d
                                                                                                                                                           \langle
                                                                                                                                                           u,|p_i||\partial_{q_i}V(q)|u\rangle
                                                                                                                                                           +\langle
                                                                                                                                                           u,|D_{p_i}||D_{q_i}|u\rangle\nonumber\\&\le
                                                                                                                                                                                                    \sum_{i=1}^d\epsilon\langle
                                                                                                                                                                                                    u,|p_i|^4u\rangle+c_{\epsilon}\langle
                                                                                                                                                                                                    u,|\partial_{q_i}V(q)|^{\frac{4}{3}}u\rangle+\epsilon\langle
                                                                                                                                                                                                    u,|D_{p_i}|^4u\rangle+c_{\epsilon}\langle 
                                                                                                                                                                                                   u,|D_{q_i}|^{\frac{4}{3}}u\rangle\label{eq3}\\&\le
                                                                                                                                                                                                                                                    c\Big(\epsilon\|O_pu\|^2_{L^2}+c_{\epsilon}\|(\sqrt{A}+K_{V})u\|^2_{L^2}\Big)
\,,
\nonumber
\end{align}
where (\ref{eq3}) is due to the Young inequality $ts\le \frac{1}{4}t^4 +\frac{3}{4}t^{\frac{3}{4}}$ for all $t,s\ge 0$ and  the last line  is  a consequence of $(\ref{eq5})$.

Therefore, combining the last inequality with
% (\ref{eq5}) and 
(\ref{eq6}), we obtain
\begin{align*}
\|(\sqrt{A}+K_{V})u\|^2_{L^2}&\ge \|(\sqrt{A}+O_p)u\|^2_{L^2}+\|X_Vu\|^2_{L^2}-c\Big[\epsilon\|O_pu\|^2_{L^2}+c_{\epsilon}\|(\sqrt{A}+K_{V})u\|^2_{L^2}\Big]\\&\ge      (1-c\epsilon)\|(\sqrt{A}+O_p)u\|^2_{L^2}+\|X_Vu\|^2_{L^2}-cc_{\epsilon}\|(\sqrt{A}+K_{V})u\|^2_{L^2}
\end{align*}
for all $u\in D(K_{V})$.

To complete the proof, it is enough to use the above inequality with (\ref{eq5}) and the fact that
\begin{align*}2\Big(A\|u\|^2_{L^2}+\|K_{V}u\|^2_{L^2}\Big)\ge \|(\sqrt{A}+K_{V})u\|^2_{L^2}
\end{align*}
for all $u\in D(K_{V})$.
\end{proof}

\begin{proof}[Proof of Theorem~\ref{thm1.2}]
If $\text{Tr}_{-}+\lambda_1\not=0$, by Proposition~\ref{prop1.1}, there exists a constant $c>0$ such that 
\begin{align*}
\|K_{V}^{-1}\|_{\mathcal{L}(L^2(\mathbb{R}^{2d}))}=\|\displaystyle\int_0^{+\infty}e^{-tK_{V}}dt\|_{\mathcal{L}(L^2(\mathbb{R}^{2d}))}&\le\displaystyle\int_0^{+\infty}\|e^{-tK_{V}}\|_{\mathcal{L}(L^2(\mathbb{R}^{2d}))}dt\\&\le\frac{c}{\sqrt{B}}~.
\end{align*}
Consequently, for all $u\in D(K_V)$,
\begin{align*}
\|u\|^2_{L^2(\mathbb{R}^{2d})}\le\frac{c}{B}\|K_V\|^2_{L^2(\mathbb{R}^{2d})}~.
\end{align*}
Combining the above inequality, along with (\ref{eq4}), one gets immediately the global subelliptic estimates 
\begin{align}
\|K_Vu\|^2_{L^2(\mathbb{R}^{2d})}\ge \frac{c}{1+\frac{A}{B}}\Big(\|O_pu\|^2_{L^2(\mathbb{R}^{2d})}&+\|X_Vu\|^2_{L^2(\mathbb{R}^{2d})}\nonumber\\&+\|\langle\partial_q V(q)\rangle^{2/3}u\|^2_{L^2(\mathbb{R}^{2d})}+\|\langle D_q\rangle^{2/3}u\|_{L^2(\mathbb{R}^{2d})}\Big)\label{2.444}
\end{align}
for all $u\in D(K_V).$

\end{proof}

\section{Subelliptic estimates with remainder for non-degenerate
  one-dimensional potentials}

 The operator $K_V$ with a potential $V(q)=\mp \frac{\nu q^{2}}{2}=-e^{2i\alpha}\frac{\nu q^2}{2}$ (where $\nu>0$ is a parameter and $\alpha\in\lbrace 0,\frac{\pi}{2}\rbrace$), is unitarily equivalent to 
 \begin{align*}
K_{\nu,\alpha}&=\frac{1}{2}(-\partial_p^2+p^2)+\Big(e^{i\alpha}\sqrt{\nu}\Big)e^{-i\alpha}\Big(p\partial_q+e^{2i\alpha}q\partial_p\Big)\\&=O_p+zX_{\alpha}
\end{align*}
where $z:=e^{i\alpha}\sqrt{\nu}$ and $X_{\alpha}:=i(e^{-i\alpha}pD_q+e^{i\alpha}qD_p)$. Actually introducing the possibly complex
parameter $z$ allows us to use the same computations for both cases
because they involve entire functions of $z\in\mathbb{C}$\,. On the other hand,
some identities make sense only when $\alpha\in\left\{0,\frac{\pi}{2}\right\}$\,, particularly those involving $O_q$ (the harmonic oscillator in $q$) or the symplectic product.
Below we sum up the cases to be studied:
$$\begin{tabular}{|l|c|r|}
  \hline
  $V(q)$ & $\alpha$ & $z$ \\
  \hline
  $\frac{-\nu q^2}{2}$ & $0$ & $\sqrt{\nu}$ \\\hline
  $\frac{+\nu q^2}{2}$ & $\frac{\pi}{2}$ & $i\sqrt{\nu}$\\
  \hline
\end{tabular}
$$
In this one dimensional case, we use the following notations:
\begin{eqnarray*}
&&O_q=\frac{1}{2}(D_q^2+q^2)\quad,\quad
O_{e^{i\alpha}q}=\frac{1}{2}(e^{-2i\alpha}D_{q}^{2}+e^{2i\alpha}q^{2})\quad,\quad
O_p=\frac{1}{2}(D_p^2+p^2)\,,\\
&&
X_{\alpha}=i(e^{-i\alpha}pD_q+e^{i\alpha}qD_p)
\quad,\quad Y_{\alpha}= i(e^{i\alpha}pq -e
^{-i\alpha} D_qD_p)\,,
\end{eqnarray*}
where $\alpha\in\left\{0,\frac{\pi}{2}\right\}$ and
$O_{e^{i\alpha}q}=e^{2i\alpha}O_{q}$ in the final applications.

The Hamilton map written as a matrix equals
\begin{align*}
H_Q&:=
\begin{pmatrix}
\mathbf{q}''_{\xi x}&\mathbf{q}''_{\xi\xi}\\ \\-\mathbf{q}''_{xx}&-\mathbf{q}''_{x\xi}
\end{pmatrix}\,,
\end{align*}
where $\mathbf{q}(q, p, \xi_q, \xi_p)$ is the
Weyl-symbol of the operator $Q$, meaning $Q =
\mathbf{q}^w(q, p, D_q, D_p)=q^{w}(x,D_{x})$\,, $x=(q,p)$: 
$$
Qu(x)=\int_{\mathbb{R}^{4d}}e^{i(x-x').\xi}q\left(\frac{x+x'}{2},\xi\right)u(x')~\frac{d\xi}{(2\pi)^{2d}}dx'\,.
$$
Noticing that 
$O_{p}$, $O_{q}$, $O_{e^{i\alpha}q}$\,, $X_{\alpha}$, $Y_{\alpha}$ and
$K_{\nu,\alpha}$ have quadratic symbols, the corresponding Hamilton maps
are written accordingly $H_{O_p}$, $H_{O_q}$, $H_{O_{e^{i\alpha}q}}$, $H_{X_\alpha}$, $H_{Y_\alpha}$ and $H_{K_{\nu,\alpha}}$\,.
Let $E:=H_{O_{e^{i\alpha}q}-O_{p}}$,
$I:=H_{-O_{e^{i\alpha}q}-O_p}$,
$J:=H_{-X_{\alpha}}$, and $K:=H_{Y_{\alpha}}$
denote respectively the Hamiltonian matrices associated to the
operators $O_{e^{i\alpha}q}-O_p$,  $-O_{e^{i\alpha}q}-O_p$, $-X_{\alpha}$ and $Y_{\alpha}$. Then one has 
\begin{align*}
&E=\begin{pmatrix}
0&0&e^{-2i\alpha}&0\\0&0&0&-1\\-e^{2i\alpha}&0&0&0\\0&1&0&0
\end{pmatrix},\quad I=\begin{pmatrix}
0&0&-e^{-2i\alpha}&0\\0&0&0&-1\\e^{2i\alpha}&0&0&0\\0&1&0&0
\end{pmatrix},
\\&
J=\begin{pmatrix}
0&-ie^{-i\alpha}&0&0\\-ie^{i\alpha}&0&0&0\\0&0&0&ie^{i\alpha}\\0&0&ie^{-i\alpha}&0
\end{pmatrix},\quad K=\begin{pmatrix}
0&0&0&-ie^{-i\alpha}\\0&0&-ie^{-i\alpha}&0\\0&-ie^{i\alpha}&0&0\\-ie^{i\alpha}&0&0&0
\end{pmatrix}~.
\end{align*}
Note that $E$ commutes with $I,J,K$ and $IJ=K$ with the relations
\begin{eqnarray}
\label{eq.IJK2}
&&E^2=I^2=J^2=K^2 =-1\;\text{for~all}\;\alpha\in\mathbb{R}\,,\\
\label{eq.IJKbar}
\text{and}&&
\overline{E}=E,\; \overline{I}=I,\;\overline{J}=-e^{2i\alpha}J,\;\overline{K}=-e^{2i\alpha}K\quad\text{when}
~\alpha\in\{0,\pi/2\}\,.
\end{eqnarray}
These relations, $IJ=K$, and \eqref{eq.IJK2} ensure that $(1,I,J,K)$ can be considered algebraically as a basis of (bi-)quaternions.
Note in particular that
\begin{eqnarray*}
&&H_{O_p}=-\frac{1}{2}(E+I) \quad,\quad H_{X_{\alpha}}=-J\\
&& H_{Y_{\alpha}} =K\quad,\quad
H_{K_{\nu,\alpha}}=-\frac{1}{2}(E+I+2zJ)
\end{eqnarray*}
for all $\alpha \in \mathbb{R}$, while the relations 
$$
\begin{pmatrix}
0&-\text{Id}_{\mathbb{R}^2}\\\text{Id}_{\mathbb{R}^2}&0
\end{pmatrix}=\sin(\alpha)E+\cos(\alpha)I
\quad,\quad H_{O_q}=\frac{e^{2i\alpha}}{2}(E-I)$$
hold for $\alpha\in \{0,\frac{\pi}{2}\}$\,.

The commutation property with the matrix $E$ can be interpreted  as follows at the operator level:
consider the two commutators 
$\Big[O_p,X_{\alpha}\Big]=iY_{\alpha}$ and
$\Big[O_{e^{i\alpha}q},X_{\alpha}\Big]=iY_{\alpha}$.
Then the operator $O_{e^{i\alpha}q}-O_p$ commutes with $O_p$ and
$X_{\alpha}$\,. Once this reduction is done, the quaternionic
structure can be guessed as well from the operator level after
computing all the commutators of $O_p$, $O_{e^{i\alpha}q}$, $X_\alpha$
and $Y_\alpha$\,.

\subsection{General estimate when  $V(q)=\pm \frac{\nu q^2}{2}$\,, $\nu>0$}
\begin{prop}\label{prop3.1}
Let $\nu>0$  be a  parameter and $\alpha\in\left\{0,\frac{\pi}{2}\right\}$. 
There exists a constant $C>0$, independent of  $\nu$, such that
\begin{align*}
\|\sqrt{\nu\;O_q}\;e^{-t(K_{\nu,\alpha}+\sqrt{\nu})}\|_{\mathcal{L}(L^2(\mathbb{R}^{2})}\le \frac{C}{t^{\frac32}}
\end{align*}

{holds} for all $t> 0$.
\end{prop}

\begin{lem}
\label{le:delta0}
One can find a function $\delta_0(t)>0$\,, specified below in \eqref{eq:delta0}\eqref{eq:A(t)}, defined in  $[0,+\infty[$ such that for all $\delta(t)\in [0,\delta_0(t)[$

 $$\|e^{\delta(t)O_q}e^{-tK_{\nu,\alpha}}\|_{\mathcal{L}(L^2(\mathbb{R}^2))}\le
 1$$ is satisfied for all $t>0$.
\end{lem}

\begin{proof}

The exact classical quantum
correspondence, valid for $Q_j = q^w_j, j = 1, 2, 3$, when $q_j$ are complex-valued quadratic forms with associated Hamilton maps $H_{Q_{j}}$ and positive Hamilton flows $\exp H_{Q_j}$ (see\cite{Hor}\cite{Vio}), says that
$$ 
\exp H_{Q_1} \exp H_{Q_2} = \exp H_{Q_3}\iff e^{-iQ_1}
e^{-iQ_2} = \pm e^{-iQ_3}\,.
$$
We will determine conditions such that the canonical transformation
\[
	\exp H_{i\delta(t) O_q}\exp H_{-itK_{\nu,\alpha}}
\]
is strictly positive in the sense defined in \eqref{eq.strpos}. Working from the Hamilton flow, one can therefore compute exactly (\cite{Vio2}, Proposition 4.8) a compact operator of the form $e^{-i Q_2}$ for $Q_2$ quadratic such that
\[
	e^{-\delta(t)O_q}e^{-iQ_2} = e^{-i\delta(t)(i^{-1}O_q)}e^{-iQ_2} = \pm e^{-it(i^{-1}K_{\nu,\alpha})}=\pm e^{-tK_{\nu, \alpha}}
\]
Applying this equality to the dense set of linear combinations of Hermite functions, this shows that $e^{-tK_{\nu, \alpha}}$ takes $L^2(\Bbb{R}^2)$ to the domain of $e^{\delta(t)O_q}$ with the estimate 
\[
	\|e^{\delta(t)O_q}e^{-tK_{\nu, \alpha}}\|_{\mathcal{L}(L^2(\mathbb{R}^2))} = \|e^{-iQ_2}\|_{\mathcal{L}(L^2(\mathbb{R}^2))} \leq 1.
\]

We will compute
$$
e^{i\delta(t)H_{O_q}}e^{-it H_{K_{\nu,\alpha}}}
$$
which will be done by using biquatertionic expressions. The
compactness of $e^{-iQ_2}$, and the fact that its norm is bounded by
$1$, is a consequence of the positivity condition \eqref{eq.strpos} which will be checked explicitly.

Set, for all $t\ge 0$, $\kappa(t)=e^{-itH_{K_{\nu,\alpha}}}$ and $\kappa_0(\delta)=e^{i\delta(t)H_{O_q}}$\,, and consider the canonical transformation $$\kappa(t):=e^{-itH_{K_{\nu,\alpha}}}=e^{i\frac{t}{2}(E+I+2zJ)}$$
for all $t\ge 0$.
Let $n_1$ denote 
\[
	n_1 = \sqrt{N(I+2zJ)}=\sqrt{1+4z^2}\neq 0 \;\text{when}\; z\neq \pm \frac{i}{2}
\]
such that $\hat{v}=\frac{I+2zJ}{n_1}$ satisfies $\hat{v}^2=-1$\,.

Using the fact that $E$ commutes with $I$ and $J$, and the formula (\ref{formula1})\;,
\begin{align}\kappa(t)&=e^{i\frac{t}{2}E}e^{i\frac{t}{2}n_1\hat{v}}=e^{i\frac{t}{2}E}\Big(\ch(\frac{tn_1}{2})+i\;\frac{\sh(\frac{tn_1}{2})}{n_1}(I+2zJ)\Big)\nonumber\\& =: e^{i\frac{t}{2}E}\Big(C(t)+i\;S(t)(I+2zJ)\Big)~.\label{eq18}
\end{align}

The functions $\mathbb{R}\ni t\mapsto C(t)$ and $\mathbb{R}\ni t\mapsto S(t)$ do not depend on the choice of
the square root $\sqrt{1+4z^2}$\,, because $\ch$ is an even function
and $\sh$ an odd function. Moreover, they are real when $z\in
\mathbb{R}\cup i\mathbb{R}$\,, which corresponds to
$z=e^{i\alpha}\sqrt{\nu}$\,, $\alpha\in \left\{0,\frac{\pi}{2}\right\}$\,.

On the other hand,
 \begin{align}
\kappa_0(\delta)&=e^{-i\delta(t)H_{O_q}}=e^{\frac{i}{2}\delta(t)e^{2i\alpha}E}e^{-\frac{i}{2}\delta(t)e^{2i\alpha}I}\nonumber\\&=e^{\frac{i}{2}\delta(t)e^{2i\alpha}E}\Big(\ch(\frac{\delta(t)}{2}e^{2i\alpha})-i\;\sh(\frac{\delta(t)}{2}e^{2i\alpha})I\Big)\label{eq19}~.
\end{align}
When  $\sigma=
\begin{pmatrix}
  0&-\text{Id}\\
\text{Id}&0
\end{pmatrix}
$ denotes the matrix of the symplectic form on $\mathbb{R}^{2\times 2}$\,,
the equality 
 $\sigma=\sin(\alpha)E+\cos(\alpha)I$ holds when 
$
\alpha\in \left\{0,\frac{\pi}{2}\right\}$ (and only in those cases $\operatorname{mod} \pi$).
{
As established in \cite{Vio2}, it is possible to write
$e^{\delta(t)O_q}e^{-tK_{\nu,\alpha}}=e^{-iQ_{2}}$
with $Q_{2}=q_{2}^{w}$\,, with $e^{-iQ_2}$ a compact operator\,, when 
 the canonical transformation $\kappa_0\kappa$ satisfies the strict positivity condition}
\begin{equation}\label{eq.strpos}
i\Big[\sigma\Big(\overline{\kappa_0\kappa z},\kappa_0\kappa z\Big)-\sigma\Big(\overline{z}, z\Big)\Big]> 0\; \text{for all }\;z\in\mathbb{C}^4\setminus\lbrace 0\rbrace{~.}
\end{equation}
This condition is equivalent to the condition that the Hermitian matrix 
$$
i\Big((\kappa_0\kappa)^*\sigma\kappa_0\kappa-\sigma\Big)=i\Big(\kappa^*\kappa_0^*\sigma\kappa_0\kappa-\sigma\Big)=i\Big(\kappa^*\kappa_0\sigma\kappa_0\kappa-\sigma\Big){~,}$$
is positive definite, or equivalently that
\begin{align*}
 \kappa_0(i\sigma)\kappa_0-(\kappa^*)^{-1}(i\sigma)(\kappa)^{-1}=\kappa_0(\delta)(i\sigma)\kappa_0(\delta)-\kappa^*(-t)(i\sigma)\kappa(-t)
\end{align*}
is positive definite.
 
Since $E$ commutes with $I$, $J$ and $K$, the spectral decomposition of $E$ allows us to study 2-by-2 matrices instead of 4-by-4 matrices: $T^*_{\pm}(iE)T_{\pm}=\pm \text{Id}~$, where 
\[
	T_{\pm}=\frac{1}{\sqrt{2}}\begin{pmatrix}1&0\\0&1\\\mp ie^{\pm2i\alpha}&0\\0&\pm i\end{pmatrix},\quad T^*_{\pm}T_{\pm} =\begin{pmatrix}
1&0\\0&1
\end{pmatrix}~.
\]

Letting $$\widetilde E:= T^*_{\pm}E T_{\pm}=\mp i \text{Id}=\begin{pmatrix}
\mp i&0\\0&\mp i
\end{pmatrix},\;\;\;\widetilde I:= T^*_{\pm}I T_{\pm}=\begin{pmatrix}
\pm i&0\\0&\mp i
\end{pmatrix},$$ 
$$\widetilde J:= T^*_{\pm}J T_{\pm}=\begin{pmatrix}
0&-ie^{-i\alpha}\\-ie^{i\alpha}&0
\end{pmatrix},\;\;\;\widetilde K:= T^*_{\pm}K T_{\pm}=\begin{pmatrix}
0&\pm e^{-i\alpha}\\\mp e^{i\alpha}&0
\end{pmatrix}~,$$
we get
\begin{align}
 T_{\pm}^* \kappa_0(\delta)(i\sigma)\kappa_0(\delta)T_{\pm}=e^{\pm\delta(t)e^{2i\alpha}}\Big[\pm \sin(\alpha)c(t)-s(t)\cos(\alpha)+i\Big(\cos(\alpha)c(t)\mp\sin(\alpha)s(t)\Big)\widetilde{I}\Big]~,\label{eq16}
 \end{align}
where $c(t)=\ch(\delta(t)e^{2i\alpha})$ and $s(t)=\sh(\delta(t)e^{2i\alpha})${~.}
Similarly,
\begin{align}T^*_{\pm}\kappa^*(-t)(i\sigma)\kappa(-t)T_{\pm}=e^{\mp t}\Big[&\pm \sin(\alpha)\Big(C^2(t)+(1-(2z)^2)S^2(t)\Big)-2\cos(\alpha)C(t)S(t)\nonumber\\&+i\cos(\alpha)\Big(C^2(t)+(1-(2z)^2)S^2(t)\Big)\widetilde{I}\nonumber\mp 2i \sin(\alpha)C(t)S(t)\widetilde{I}\\&+4zi\cos(\alpha)S^2(t)\widetilde{J}\nonumber\mp 4z\sin(\alpha)S^2(t)\widetilde{K}\Big]{~.}\\&
\label{eq17}
\end{align}
Taking into account (\ref{eq16}) and (\ref{eq17}), 
\begin{align*}
T_{\pm}^*\kappa_0(\delta)(i\sigma)\kappa_0(\delta)T_{\pm}-T^*_{\pm}\kappa^*(-t)(i\sigma)\kappa(-t)
&T_{\pm}=\pm
           e^{\pm\delta(t)e^{2i\alpha}}\Big( \sin(\alpha)c(t)\mp
           \cos(\alpha)s(t)\Big)
\\
&+e^{\mp t}\Big(\mp
           \sin(\alpha)(1+2S^2)+2\cos(\alpha)CS\Big)
\\&
+i\Big[e^{\pm\delta(t)e^{2i\alpha}}\Big(\cos(\alpha)c(t)\mp\sin(\alpha)s(t)\Big)
\\
&\quad-
    e^{\mp t}\Big(\cos(\alpha)(1+2S^2(t))\mp
    2\sin(\alpha)C(t)S(t)\Big)\Big]\widetilde{I}
\\&
-4zie^{\mp t}\cos(\alpha)S^2(t)\widetilde{J}\pm4ze^{\mp
    t}\sin(\alpha)S^2(t)\widetilde{K}
\\=&e^{\mp t}\Big(a+b\widetilde{I}+c\widetilde{J}+d\widetilde{K}\Big)~.
\end{align*}

The determinant of the Hermitian matrix $e^{\pm t}\Big(T_{\pm}^* \kappa_0(\delta)(i\sigma)\kappa_0(\delta)T_{\pm}-T^*_{\pm}\kappa^*(-t)(i\sigma)\kappa(-t)T_{\pm}\Big)$ is equal to
\begin{align*} 
a^2+b^2+c^2+d^2=1-e^{\pm t}(2+4S^2-e^{\pm t})\mp e^{\pm t}(1-e^{\pm 2\delta(t)e^{2i\alpha}})\Big(2CS\mp (1+2S^2-e^{\pm t} )\Big)~.
\end{align*}
Let $\delta_0(t)>0$ be the function  which cancels the determinant, or equivalently for which one has, for all $t>0$,

\begin{align*} 
2\Big(2S^2-(\ch(t)-1)\Big)=\mp (1-e^{\pm 2\delta(t)e^{2i\alpha}})\Big(2CS+\sh(t)\mp \Big(2S^2-(\ch(t)-1)\Big)\Big){~.}
\end{align*}

After some computation, we find that this function is independent of the sign in the expression above and is given by  
\begin{eqnarray}
\label{eq:delta0}
&&\delta_{0}(t)=
   \frac{e^{-2i\alpha}}{2}\ln\Big(1-\frac{2A(t)}{2C(t)S(t)+\sh(t)+A(t)}\Big){~,}\\
\label{eq:A(t)}
\text{where}&&
A(t):=\Big(2S^2(t)-(\ch(t)-1)\Big)\,.
\end{eqnarray}
We know that, when $\delta = 0$ and $\alpha \in \{0, \frac{\pi}{2}\}$,
the Hamilton flow $\kappa(t)$ is positive because $e^{-tK_{\nu, 0}}$
is a compact operator (see \cite{HeNi}\cite{HiPr}).
 By connectedness of the set of positive definite hermitian matrices
and because the result holds for $\delta(t)=0$, the flow $\kappa_0(\delta)\kappa(t)$ is a positive canonical transformation so long as the determinant is positive on $[0, \delta]$. Therefore
 $\delta(t)\in[0,\delta_0(t) [$ implies
\begin{align*}
\|e^{\delta(t)O_q}e^{-tK_{\nu,\alpha}}\|_{\mathcal{L}(L^2(\mathbb{R}^2))}\le 1~,
\end{align*}
because any such compact Schr\"odinger evolution has norm less than 1
(see \cite{Vio2}).
\end{proof}

\begin{proof}[Proof of Proposition~\ref{prop3.1}]
When $0<\epsilon_0<1$ , there exists a constant $c>0$  independent of
$\nu$ such that 
$$\delta_0(t)\ge c\nu t^3$$ holds for all $0<t\le
t_0:=\frac{\epsilon_0}{1+|z|}=\frac{\epsilon_{0}}{1+\sqrt{\nu}}$\,.
This can be seen via the expansion 
$$
\frac{1}{2}
\ln\left(1-\frac{2A(t)}{A(t)+2C(t)S(t)+\sh(t)}\right)
=
\frac{z^2}{12}  t^3+\mathcal{O}\left((1+|2z|^2)^2t^5\right)~,
$$
which is uniform with respect to the parameter $\nu$ for all $t\in ]0,t_0]$\,. 

We write the quantity $\|\sqrt{\nu\;O_q}e^{-t(K_{\nu,\alpha}+\sqrt{\nu})}\|_{\mathcal{L}(L^2(\mathbb{R}^2))}$ in the form
\begin{align*}
& \left\{
\begin{array}{ll}
\|\sqrt{\frac{\nu}{\delta(t)}}\sqrt{\delta(t)\;O_q}\;e^{-\delta(t)O_q}
e^{\delta(t)O_q}
      e^{-tK_{\nu,\alpha}}e^{-t\sqrt{\nu}}\|_{\mathcal{L}(L^2(\mathbb{R}^2))}\;\;\;\;\text{if}\;\;\;0<t\le
      t_0\,,
\\\\
\|\sqrt{\nu}
      e^{-t\sqrt{\nu}}\sqrt{O_q}
\;e^{-\delta(t_0)O_q}
e^{\delta(t_0)O_q} 
e^{-t_0K_{\nu,\alpha}}e^{-(t-t_0)K_{\nu,\alpha}}\|_{\mathcal{L}(L^2(\mathbb{R}^2))}
\;\;\;\;\text{if}\;\;\;t\ge t_0\,,
\end{array}
\right.\end{align*}
where Lemma~\ref{le:delta0} is applied with $\delta(t)=\frac{\delta_{0}(t)}{2}$\,.
For both cases we get the upper bounds
\begin{align*}
\left\{
    \begin{array}{ll}
\underbrace{ \sqrt{\frac{\nu}{\delta(t)}}}_{\le
      \frac{\sqrt{2}}{\sqrt{c}t^{\frac32}}} 
\underbrace{\|\sqrt{\delta(t)O_q}e^{-\delta(t)O_q}\|}_{\le
      c}.
\underbrace{\|e^{\delta(t)O_q}
      e^{-tK_{\nu,\alpha}}\|}_{\le 1}.
\underbrace{e^{-t\sqrt{\nu}}}_{\le 1}
\;\;\;\;\text{if}\;\;\;0<t\le t_0\,,
\\\\
(1+\sqrt{\nu})^{\frac{3}{2}}
      e^{-t\sqrt{\nu}}
\underbrace{\|\frac{\sqrt{\nu}}{(1+\sqrt{\nu})^{\frac 3 2}}
\sqrt{O_q}e^{-\delta(t_0)O_q}\|}_{\le
      \frac{\sqrt{\nu}}{(1+\sqrt{\nu})^{\frac 3 2}\sqrt{\delta(t_0)}}}.
\underbrace{\|e^{\delta(t_0)O_q}
      e^{-t_0K_{\nu,\alpha}}\|}_{\le 1}.
\underbrace{\|e^{-(t-t_0)K_{\nu,\alpha}}\|}_{\le e^{-\frac{(t-t_{0})}{2}}}
\;\;\;\;\text{if}\;\;\;t\ge t_0\,.
\end{array}
\right.
\end{align*}
For the second case $t\geq t_{0}$\,, we use
$$
(1+\sqrt{\nu})^{\frac{3}{2}}
       e^{-t(\frac{t}{2}+\sqrt{\nu})}\times
       \frac{\sqrt{\nu}}{(1+\sqrt{\nu})^{\frac 3 2}\sqrt{\delta(t_0)}}
 \times e^{\frac{t_{0}}{2}}
\leq
\frac{c_{0}}{t^{3/2}}\times
\frac{\sqrt{2\nu}}{
(1+\sqrt{\nu})^{3/2}\sqrt{c\nu\frac{\epsilon_{0}^{3}}{(1+\sqrt{\nu})^{3}}}}
\times e^{\frac{\varepsilon_{0}}{2}}\leq \frac{c'_{0}}{t^{\frac{3}{2}}}\,.
$$
This ends the proof of Proposition~\ref{prop3.1} and gives
 \begin{align}
\|\sqrt{\nu}(|D_q|+|q|)e^{-t(K_{\nu,\alpha}+\sqrt{\nu})}\|_{L^2}\le \frac{C}{t^{\frac{3}{2}}}
\end{align}
for all $t>0$.

\end{proof}

\subsection{Improved remainder, case $V(q)=\frac{\nu q^2}{2},\; \nu\gg 1$}

In this section we follow the explicit 
methods  of  Aleman and Viola in \cite{Vio}\cite{AlVi}. Following
\cite{HSV}\cite{HPV} it makes use of an FBI
transform, which in this specific case is nothing but the usual Bargmann
transform
$$
B_{2}u(z)=\frac{1}{2^{2/2}\pi^{(3\times 2)/4}}\int_{\mathbb{R}^{2}}e^{-\frac{(z-y)^{2}-z^{2}/2}{2}}u(y)~dy
$$ 
with $B_{2}:L^{2}(\mathbb{R}^{2},dy)\to
L^{2}(\mathbb{C}^{2}; e^{-\frac{|z|^{2}}{2}}L(dz))\cap
\text{Hol}(\mathbb{C}^{2})$ unitary. 
\begin{lem}
\label{le:3.1}

For $\nu>\frac{1}{4}$\,, the adjoint operator
 \begin{align*}K^*_{\nu,\frac{\pi}{2}}=\frac{1}{2}(-\partial_p^2+p^2)-\sqrt{\nu}\Big(p\partial_q-q\partial_p\Big)=O_p-\sqrt{\nu}X_{\frac{\pi}{2}}\end{align*}
 is tranformed via the Bargmann transform $B_{2}$ into
%$Q(q,p,\xi_{q},\xi_{p})=\frac{\xi_{p}^{2}+p^{2}}{2}-i\sqrt{\nu}(p\xi_{q}-q\xi_{p})$\,. 
$$
B_{2}(K_{\nu,\frac{\pi}{2}})^{*}B_{2}^{*}={}^{t}zM\partial_{z}\quad,\quad M=
\begin{pmatrix}
  0&-\sqrt{\nu}\\
\sqrt{\nu}&1
\end{pmatrix}\,.
$$
and
$$
[B_{2}(e^{-tK_{\nu,\frac{\pi}{2}}}u)](z)=(B_{2}u)(e^{-tM}z)\,.
$$
\end{lem}

\begin{proof}
{Although it may be proved by a direct computation,  it is instructive
as an illustration of the general method to follow the lines of
\cite{AlVi} or \cite{Vio}, Example 2.7. Remember that it is made in essentially two
steps~: 1) Write the operator, up to an additive constant, in the ``supersymmetric'' form
${}^{t}(D_{x}-A_{+}x)B(D_{x}-A_{+}x)$ after some real canonical
transformation in $\mathbb{R}^{2d}$ (here $d=2$); 2) transform the
supersymmetric form into $i{}^{t}zM\zeta$ after 
some linear complex canonical transformation
associated with an FBI-transform.\\
\textbf{Step~1:}
The two variables $(q,p)$ are
gathered in  the notation $x=(q,p)\in
\mathbb{R}^{2}$\,, with dual variable $\xi=(\xi_{q},\xi_{p})\in
\mathbb{R}^{2}$\,.}
The hamiltonian matrix associated to $K^*_{\nu,\frac{\pi}{2}}$ is given by 
 
 $$H_{K^*_{\nu,\frac{\pi}{2}}}= \begin{pmatrix}
 0&-i\sqrt{\nu}&0&0\\i\sqrt{\nu}&0&0&1\\0&0&0&-i\sqrt{\nu}\\0&-1&i\sqrt{\nu}&0
\end{pmatrix}~. $$ 

Set $\lambda_{\epsilon_1,\epsilon_2}=\frac{\epsilon_1 i+\epsilon_2 in_1}{2}$ the eigenvalues of $H_{K^*_{\nu,\frac{\pi}{2}}}$ with their associated eigenvectors $$^tX_{\epsilon_1,\epsilon_2}=\Big(1,\frac{i\lambda_{\epsilon_1,\epsilon_2}}{\sqrt{\nu} },\frac{(\lambda_{\epsilon_1,\epsilon_2})^2-\nu}{\lambda_{\epsilon_1,\epsilon_2}},i\frac{(\lambda_{\epsilon_1,\epsilon_2})^2-\nu}{\sqrt{\nu}}\Big)~,$$
where $\epsilon_1,\epsilon_2 \in \{\pm 1\}.$ In the case $\alpha=\frac{\pi}{2}$, one has $n_1=\sqrt{1-4\nu}=i\sqrt{4\nu-1}$ for $\nu>\frac{1}{4}$\,.

As a first step we need to determine the following two spaces:
\[
\Lambda_{-}=\underset{\textrm{Im}\;\lambda<0}\bigoplus \ker(H_{K^*_{\nu,\frac{\pi}{2}}}-\lambda I)=\left\{\begin{pmatrix}
x\\A_{-}x\end{pmatrix},\;\;\;x\in\mathbb{C}^2\right\}
\]
and
\[
\Lambda_{+}=\underset{\textrm{Im}\;\lambda>0}\bigoplus \ker(H_{K^*_{\nu,\frac{\pi}{2}}}-\lambda I)=\left\{\begin{pmatrix}
x\\A_{+}x\end{pmatrix},\;\;\;x\in\mathbb{C}^2\right\},
\]
where $A_{+}$ and $A_{-}$ are two matrices in $\mathbb{M}_2(\mathbb{C})$ satisfying ${}^tA_{\pm}=A_{\pm}$ and $\pm \textrm{Im}(A_{\pm})>0$. 

The  matrix $A_{+}$ is  given by $A_{+}=B_{1+}^{-1}B_{2+}$ where $$ B_{1+}=\begin{pmatrix}
1&\frac{-1-n_1}{2\sqrt{\nu}}\\\\1&\frac{-1+n_1}{2\sqrt{\nu}}
\end{pmatrix}\;\;\;\;\text{and}\;\;\;B_{2+}=\begin{pmatrix}
i&\frac{-i-in_1}{2\sqrt{\nu}}\\\\i&\frac{-i+in_1}{2\sqrt{\nu}}
\end{pmatrix}~,$$
so $A_{+}=i\text{Id}~.$ Similarly, $A_{-}=B_{1-}^{-1}B_{2-}$ with $$ B_{1-}=\begin{pmatrix}
1&\frac{1+n_1}{2\sqrt{\nu}}\\\\1&\frac{1-n_1}{2\sqrt{\nu}}
\end{pmatrix}\;\;\;\;\text{and}\;\;\;B_{2-}=\begin{pmatrix}
-i&\frac{-i+in_1}{2\sqrt{\nu}}\\\\-i&\frac{-i-in_1}{2\sqrt{\nu}}
\end{pmatrix}~,$$
so $A_{-}=-i\text{Id}$\,. This means, after
\cite{Vio}~formula~(2.3), that the real canonical transformation on
$\mathbb{R}^{4}$ is nothing but the identity.

Hence it suffices to write $K^*_{\nu,\frac{\pi}{2}}$ in
the form 
$$
K^*_{\nu,\frac{\pi}{2}}=\;^t(D_x-A_{+}x)B(D_x-A_{-}x)~,
$$ for all $x=(q,p)\in \mathbb{R}^2$\,, where the matrix $B$
is found by identification of the two sides:
  $$B=\begin{pmatrix}
0&\frac{-\sqrt{\nu}}{2}\\\frac{\sqrt{\nu}}{2}&\frac{1}{2}
\end{pmatrix}{~.}$$
\textbf{Step~2:} Once $A_{+}$ and 
$A_{-}$ are known, the complex canonical transformation is given by
$$\kappa=\begin{pmatrix}
1&-i\\-(1-iA_+)^{-1}A_{+}&(1-iA_+)^{-1}
\end{pmatrix}{~,} $$
with associated quadratic 
phase $\varphi_{A_{+}}:\mathbb{C}^2\times \mathbb{C}^2\to
\mathbb{C}$
$$
\varphi_{A_+}(x,y)=\frac{i(x-y)^2}{2}-\frac{1}{2}\Big(x,(1-iA_+)^{-1}A_+x\Big)=i\left[\frac{(x-y)^{2}}{2}-\frac{x^{2}}{4}\right]\,,
$$
which is the one entering in the definition of the associated FBI
transform (which is $B_{2}$)\,.
The computation of $B_{2}K^*_{\nu,\frac{\pi}{2}}B_{2}^{*}$ then comes
from Egorov's theorem
\begin{eqnarray*}
  &&K^*_{\nu,\frac{\pi}{2}}(\kappa^{-1}Z)=^tZ^t\kappa^{-1}\begin{pmatrix}
-A_{+}\\\text{Id}
\end{pmatrix}B(-A_-, \text{Id})\kappa^{-1}Z
=i^tzM\zeta\\
\text{with}&&
M=(1-iA_+)B=2B=\begin{pmatrix}
0&-\sqrt{\nu}\\\sqrt{\nu}&1
\end{pmatrix}\;.
\end{eqnarray*}
The weight $e^{-2\phi(z)}L(dz)$ occuring in
the range of $B_{2}$ is $\phi(z)=\frac{|z|^{2}}{4}$ which is coherent
 with the formulas (2.6) and (2.7) of \cite{Vio}\,,
$\phi(x)=\frac{1}{4}\Big(|x|^2-\;^txCx\Big)$ because $C=(1-iA_+)^{-1}(1+iA_+)=0\,.$
\end{proof}
\begin{lem}\label{lem3.3} There exists a constant $c>0$ independent of  $\nu>1$, such
that
for all $t>0$ and all $u\in L^{2}(\mathbb{R}^{2})$\,,
$u_{t}=e^{-t(K_{\nu,\frac{\pi}{2}}+\nu^{1/3})}u$ satisfies
\begin{equation}
\frac{\nu}{2}\left(\|u_{t}\|_{L^{2}(\mathbb{R}^{2})}^{2}
+\|D_{q}u_{t}\|_{L^{2}(\mathbb{R}^{2})}^{2}+\|qu_{t}\|_{L^{2}(\mathbb{R}^{2})}^{2}
\right)=
\|\sqrt{\nu}(\frac{-\partial_q+q}{\sqrt{2}})e^{-t(K_{\nu,\frac{\pi}{2}}+\nu^{\frac{1}{3}})}u\|_{L^2(\mathbb{R}^{2})}^{2}
\le \frac{c}{t^{3}}\|u\|_{L^{2}(\mathbb{R}^{2})}^{2}\,.
\end{equation}
\end{lem}
\begin{proof}
Set $a_{q}=\frac{\partial_{q}+q}{\sqrt{2}}$ and
$a_{q}^{*}=\frac{-\partial_{q}+q}{\sqrt{2}}$ so that
$
a_{q}a_{q}^{*}=a_{q}^{*}a_{q}+1=\frac{1}{2}(D_{q}^{2}+q^{2}+1)$\,.
The identity
\begin{multline*}
\nu\|a_{q}^{*}e^{-t(K_{\nu,\frac{\pi}{2}}+\nu^{1/3})}u\|_{L^{2}(\mathbb{R}^{2})}^{2}
=\nu\|e^{-t(K_{\nu,\frac{\pi}{2}}+\nu^{1/3})}u\|_{L^{2}(\mathbb{R}^{2})}^{2}+\nu\|a_{q}e^{-t(K_{\nu,\frac{\pi}{2}}+\nu^{1/3})}u\|_{L^{2}}^{2}
\\ \leq \nu e^{-t\nu^{1/3}}\|u\|_{L^{2}(\mathbb{R}^{2})}^{2}+
\nu\|a_{q}e^{-t(K_{\nu,\frac{\pi}{2}}+\nu^{1/3})}u\|_{L^{2}}^{2}
\end{multline*}
reduces the problem to that of estimating
$\|\sqrt{\nu}a_{q}e^{-t(K_{\nu,\frac{\pi}{2}}+\nu^{1/3
})}\|.$ By taking the adjoint, it suffices to prove that
\begin{align}
\|\sqrt{\nu}e^{-t(K^*_{\nu,\frac{\pi}{2}}+\nu^{1/3})}a_q^*f\|_{L^2(\mathbb{R}^2)}\le \frac{c}{ t^{\frac{3}{2}}}\|f\|_{L^2(\mathbb{R}^2)}\label{eq7}
\end{align}
is satisfied for all $f\in L^{2}(\mathbb{R}^{2},dqdp)$ and for all $t>0$\,.

Conjugating by the Bargmann transform $B_{2}$, the creation operator
$B_2a_q^*B_2^*=B_2(\frac{-\partial_q+q}{\sqrt{2}})B_2^*=\frac{z_{q}}{\sqrt{2}}\times$
is nothing but multiplication by the complex component $z_{q}$ in $\mathbb{C}^{2}=\mathbb{C}_{q}\times\mathbb{C}_{p}$\,. The inequality 
(\ref{eq7}) is therefore equivalent to 
\begin{align}
\|\sqrt{\nu}e^{-t(Mz\partial_z+\nu^{1/3})}z_qu\|_{H_{\phi}}\le \frac{c}{ t^{\frac{3}{2}}}\|u\|_{H_{\phi}} 
\end{align}
for all $u\in H_{\phi}=L^{2}(\mathbb{C}^{2},e^{-	\frac{|z|^{2
    }}{2}}L(dz))\cap \text{Hol}(\mathbb{C}^{2})$\,, with $\phi(z)=\frac{|z|^{2}}{4}$\,.

Let $u\in H_{\phi}$~, setting $v(z)=z_qu(z)$, one has  $e^{-tMz\partial_z}v(z)=v(e^{-tM}z)$ and it follows that
\begin{align*}
\|e^{-tMz\partial_z}z_qu\|^2_{H_{\phi}}&=\int_{\mathbb{C}^2}|v(e^{-tM}z)|^2|(e^{-tM}z)_{q}|^{2}e^{-2\phi(z)}L(dz)\\&=e^{2t\operatorname{Tr} M}\int_{\mathbb{C}^2}|v(z')|^2|z'_q|^2e^{-\phi(z')}e^{-2[\phi(e^{tM}z')-\phi(z')]}L(dz'){~.}\end{align*}
So our problem is reduced to the proof of the existence of a constant $c>0$ that does not depend on  $\nu$ such that  \begin{align*}
\sup\limits_{z\in\mathbb{C}^2}|z_q|^2e^{-\frac{1}{2}\Big(|e^{tM}z|^2-|z|^2\Big)}e^{-t\nu^{1/3}}\le \frac{c}{\nu t^3}
\end{align*}
for all $t>0$.

Let us start by checking that  $z\mapsto \phi(e^{tM}z)-\phi(z)$
defines a positive definite hermitian form for $t>0$\,.

From the expression given in Lemma~\ref{le:3.1}\,,  $M$ is
easily written in terms of Pauli's matrices:
\begin{eqnarray*}
    &&M=\frac{1}{2}\text{Id} -\frac12\sigma_3-i\sqrt{\nu}\sigma_2~,\\
\text{with}&&
\sigma_1=\begin{pmatrix}
0&1\\1&0
\end{pmatrix}~;~\sigma_2=\begin{pmatrix}
0&-i\\i&0
\end{pmatrix}\;\;\;\text{and}\;\;\;\sigma_3=\begin{pmatrix}
1&0\\0&-1
\end{pmatrix}~.
  \end{eqnarray*}
Recall that Pauli's matrices are involutory:
$$
\sigma_1^2 = \sigma_2^2 = \sigma_3^2 = -i \sigma_1 \sigma_2 \sigma_3 = \text{Id}~\,,$$
and  that $(\text{Id},-i\sigma_1,-i\sigma_2,-i\sigma_3)$ can be interpreted as a
basis  of (bi)quaternions.

Using formula (\ref{formula1}), one has for all $t>0$

$$e^{tM}=e^{\frac{t}{2}}\Big(C(t)+2S(t)(-\frac12 \sigma_3-i\sqrt{\nu}\sigma_2)\Big)~.$$
From this, we compute 
\begin{align*}
(e^{tM})^*e^{tM}&=e^{t}\Big(1+2S^2(t)-2C(t)S(t)\sigma_3-4\sqrt{\nu}S^2(t)\sigma_1\Big)\\&=e^{t}(a+v)~,
\end{align*}
with $a=1+2S^2(t)$ and $v=-2C(t)S(t)\sigma_3-4\sqrt{\nu}s_1^2(t)\sigma_1~$.

The eigenvalues of $(e^{tM})^*e^{tM}$ are given by
$$\lambda_{\pm}=e^{t}(a\pm \sqrt{-N(v)})~,$$
where
$N(v)=-\Big(2C(t)S(t)\Big)^2-\Big(4\sqrt{\nu}S^2(t)\Big)^2=-4S^{2}-4S^{4}<0$
owing to $(4\nu-1)S^{2}+C^{2}=1$\,,
and where $\sqrt{-N(v)}$ is the usual square root.

In order to prove that the hermitian form
$z\mapsto\phi(e^{tM}z)-\phi(z)=\;^t\bar{z}\Big((e^{tM})^*(e^{tM})-\text{Id}\Big)z$
is positive definite, it suffices to 
check  $\lambda_{-}>1$ for all $t>0$, $\lambda_+$ being clearly
strictly larger than $1$\,.
The eigenvalue $\lambda_{-}$ equals
$$
\lambda_{-}=e^{t}(1+(1+S^{2})-2|S|\sqrt{1+S^{2}})
=e^{t}e^{-2\operatorname{Argsh} |S|}
$$
which is larger than $1$ if and only if $\sh(t/2)-|S|>0$ or
$[\sh(t/2)-S(t)][\sh(t/2)+S(t)]>0$ because $\sh(t/2)>0$\,.
This is true since both factors vanish at $t=0$ with a positive
derivative  for $t>0$
owing to $\ch(t/2)>1>\pm\cos(\frac{t}{2}\sqrt{4\nu-1})$\,.\\

Now denote \begin{align*}r_1=\sqrt{4\nu-1}~;\;\;\;Q_t(z)=\;^t\bar{z}\Big[(e^{tM})^*(e^{tM})-\text{Id}\Big]z\;\;\;\text{ and }\;\;\;S_t(z_1,z_2)=\;^t\bar{z_1}\Big[(e^{tM})^*(e^{tM})-\text{Id}\Big]z_2
\end{align*}
for all $z=(z_1,z_2)\in \mathbb{C}\times\mathbb{C}$. Writing $z_q=l(z)$ where $l$ is a linear form with kernel $ker\;l=\mathbb{C}e_p$, $\Big(\mathbb{C}^2=\mathbb{C}e_q\oplus \mathbb{C}e_p$ where $e_q=(1,0)$ and $e_p=(0,1)\Big)$, we construct an orthonormal basis $(e_q',e_p)$ for $Q_t$ with \begin{align*}e_q'=e_q-\frac{S_t(e_p,e_q)}{S_t(e_p,e_p)}e_p=\begin{pmatrix}
1\\\\\ \frac{4\sqrt{\nu}\;S^2(t)}{(1-e^{-t})+2S^2(t)+2S(t)C(t)}
\end{pmatrix}\;,
\end{align*} 
where 
$$\left\{
\begin{array}{l}S_t(e_p,e_q)=-4\sqrt{\nu}\;e^t \;S^2(t)\\\\ S_t(e_p,e_p)=e^t\Big[(1-e^{-t})+2 S^2(t)+2S(t)C(t)\Big]\;.\end{array}\right.
$$
 
In this new basis, 
$z=\alpha e'_q+\beta e_p$ then $l(z)=\alpha\;l(e'_q)$ and $Q_t(z)=|\alpha|^2Q_t(e'_q)+|\beta|^2Q_t(e_p)$. This gives immediately 
\begin{align*}
|z_q|^2e^{\frac{-Q_t(z)}{2}}=|\alpha|^2|l(e'_q)|^2e^{\frac{-|\alpha|^2Q_t(e'_q)-|\beta|^2Q_t(e_p)}{2}}.
\end{align*}
and then 
\begin{align*}
\sup\limits_{z\in\mathbb{C}^2}|z_q|^2e^{-\frac{1}{2}\Big(|e^{tM}z|^2-|z|^2\Big)}=\sup\limits_{s\in\mathbb{R}_{+}}|l(e'_q)|^2e^{-\frac{s Q_t(e'_q)}{2}}=\frac{2|l(e'_q)|^2}{Q_t(e'_q)}\sup\limits_{\sigma\in \mathbb{R}_{+}}\sigma  e^{-\sigma}=c_0\frac{2|l(e'_q)|^2}{Q_t(e'_q)}=c_0\frac{2}{Q_t(e'_q)}
\end{align*}
where $c_0=\sup\limits_{\sigma\in \mathbb{R}_{+}}\sigma e^{-\sigma}$ and 

\begin{align*}
Q_t(e'_q)&=S_t(e'_q,e'_q)=\frac{4\Big(\sh^2(\frac{t}{2}) -S^2(t)\Big)}{(1-e^{-t})+2S^2(t)+2S(t)C(t)}~.
\end{align*}
Recall that, in the case $\alpha=\frac{\pi}{2}$ and for $\nu>\frac{1}{4}$, we define $C(t)=\cos(\frac{tr_1}{2})$ and $S(t)=\frac{\sin(\frac{tr_1}{2})}{r_1}$.

All that remains is to control the following quotient for all $t> 0$:
\begin{align*}\frac{1}{Q_t(e'_q)}=\frac{(1-e^{-t})+2S(t)\Big(S(t)+C(t)\Big)}{4\Big[\sh^2(\frac{t}{2})-S^2(t)\Big]}:=\frac{N}{D}~.
\end{align*}
$\bullet$ Starting with the case when  $t\ge \frac{4}{r_1}$,
\begin{align*}N=(1-e^{-t})+2S(t)\Big(S(t)+C(t)\Big)\le 1+\frac{4}{r_1}\le 2~.
\end{align*}
On the other hand, \begin{align*}|&S(t)|\le \frac{1}{r_1}\le \frac{t}{4}\le \frac{1}{2}\sh(\frac{t}{2})\;\;\;\;\; \text{implies}\;\;\;\; D\ge \sh^2(\frac{t}{2})~.\end{align*}
Then \begin{align*}
\frac{1}{Q_t(e'_q)}\le \frac{2}{\sh^2(\frac{t}{2})}\le 2e^{-t}
\end{align*}
for all $ t\ge \frac{4}{r_1}$.
\\ 
$\bullet$ Now observe that for $t\le \frac{4}{r_1}$, one has the following two expansions:
\begin{align*}\sh(\frac{t}{2})+S(t)=\sum_{k=0}^{+\infty}(-1)^k(r_1^{2k}+(-1)^k)\frac{t^{2k+1}}{2^{2k+1}(2k+1)!}\end{align*}
and
\begin{align*}
\sh(\frac{t}{2})-S(t)=\sum_{k=0}^{+\infty}(-1)^k(-r_1^{2k}+(-1)^k)\frac{t^{2k+1}}{2^{2k+1}(2k+1)!}{~.}
\end{align*} 
Furthermore, 
\begin{align*}\left|\sh(\frac{t}{2})+S(t)-t-\frac{r_1^2-1}{48}t^3\right|\le \frac{(r_1^4+1)t^5}{2^5\times120}
\end{align*}
which implies
 \begin{align}\frac{1}{t}\Big(\sh(\frac{t}{2})+S(t)\Big)&\ge 1-\frac{(r_1^2-1)}{48}t^2-\frac{r_1^4+1}{2^5\times 120}t^4 \nonumber\\&\ge 1-\frac{16}{48}-\frac{2\times 4^4}{2^5\times 120}\ge 1-\frac{1}{3}-\frac{2}{15}=\frac{8}{15}~.\label{eq14}
 \end{align}
 
Similarly,
 \begin{align*}
 \Big| \sh(\frac{t}{2})-S(t)-\frac{r_1^2+1}{48}t^3\Big|&\le \frac{(r_1^4-1)t^5}{2^5\times 120}=\frac{(r_1^2+1)t^3}{48}\times \frac{(r_1^2-1)t^3}{4\times 20}\\&\le \frac{(r_1^2+1)t^3}{48}\frac{(r_1t)^2}{4\times 20}\\&\le \frac{(r_1^2+1)t^3}{48}\frac{1}{5}~,\end{align*}
which gives
 \begin{align}\;\;\;\;\;\;\;\;\;\;\;\;\;\;\;\; \sh(\frac{t}{2})-S(t)\ge \frac{(r_1^2+1)t^3}{48}\times \frac{4}{5}~.\label{eq15}
 \end{align}
Taking into account (\ref{eq14}) and (\ref{eq15}) we get
\begin{align*}D\ge \Big(\sh(\frac{t}{2})+S(t)\Big)\Big(\sh(\frac{t}{2})-S(t)\Big)\ge t\times\frac{8}{15}\times \frac{(r_1^2+1)t^3}{48}\times \frac{4}{5}{~.}\end{align*}
On the other hand,
\begin{align*}
N=(1-e^{-t})+2S(t)\Big(S(t)+C(t)\Big)&=2t+(1-\frac{r_1^2}{6})t^3-\frac{1+r_1^2}{24}t^4+\mathcal{O}(r_1^4t^5)\\&=t\Big(2+(1-\frac{r_1^2}{6})t^2-\frac{1+r_1^2}{24}t^3+\mathcal{O}((r_1t)^4)\Big)~.
\end{align*}
Hence $ N\le ct $ for all  $t\le \frac{4}{r_1}$
and  \begin{align*}\frac{1}{Q_t(e'_q)}=\frac{N}{D}\le \frac{c}{\nu t^3}\;\;\;\text{for all}\;\;\;t\le \frac{4}{r_1}{~.}
\end{align*}
Thus there exists a constant $c>0$ such that, for all $u\in H_{\phi}$,
\begin{align*}\displaystyle \|e^{-tMz\partial_z}z_qu\|^2_{H_{\phi}}\le  \left\{
    \begin{array}{ll}
      \frac{c}{\nu t^3}\|u\|^2_{H_{\phi}} \;\;\;\;\;\text{for all }\;\;t\le \frac{4}{r_1}\\ce^{-t}\|u\|^2_{H_{\phi}} \;\;\;\;\;\text{ for all}\;\;t\ge \frac{4}{r_1}
    \end{array}
\right.
\end{align*}
which is equivalent to
\begin{align*}\displaystyle \|e^{-tK^*_{\nu,\frac{\pi}{2}}}a^*_qv\|_{L^2}\le  \left\{
    \begin{array}{ll}
      \frac{c}{\sqrt{\nu t^3}}\|v\|_{L^2} \;\;\;\;\;\text{ for all }\;\;t\le \frac{4}{r_1}\\ce^{-t}\|v\|_{L^2} \;\;\;\;\;\text{ for all }\;\;t\ge \frac{4}{r_1}
    \end{array}
\right.
\end{align*}
for all $v\in D(K_{\nu,\frac{\pi}{2}})$.

From this, we deduce that

\begin{align*}\displaystyle \|a_qe^{-t(\nu^{\frac{1}{3}}+K_{\nu,\frac{\pi}{2}})}v\|_{L^2}\le  \left\{
    \begin{array}{ll}
      \frac{c}{\sqrt{\nu t^3}}\|v\|_{L^2} \;\;\;\;\;\text{if }\;\;t\le \frac{4}{r_1}\\ce^{-\nu^{\frac{1}{3}}t}\|v\|_{L^2} \;\;\;\;\;\text{if }\;\;t\ge \frac{4}{r_1}
    \end{array}
\right.\end{align*}
for every $v\in D(K_{\nu,\frac{\pi}{2}})$. When $0 < t \leq \frac{4}{r_1}$, we clearly have
\[
	\|\sqrt{\nu} a_q e^{-t(\nu^{\frac{1}{3}} + K_{\nu, \frac{\pi}{2}})}\|_{\mathcal{L}(L^2(\mathbb{R}^2))} \leq \frac{C}{t^{\frac{3}{2}}}~.
\]
When $t \geq \frac{4}{r_1}$, we obtain the same result by writing
\[
	c\sqrt{\nu}e^{-\nu^{\frac{1}{3}} t} = \frac{c}{t^{\frac{3}{2}}}\left(\nu^{\frac{1}{3}} t\right)^{\frac{3}{2}}e^{-\nu^{\frac{1}{3}} t}
\]
and noting that the function $s^{3/2}e^{-s}$ is bounded on $[0, \infty)$. This establishes the inequality for all $t > 0$ and completes the proof of the lemma.
\end{proof}

\section{Resolvent estimates when $V(q)=-\frac{\nu q^2}{2}, \nu \gg 1$ }

In this section, we use the same notations as in the previous one and we take $\alpha=0$. Giving the exact norm of the semigroup $e^{-tK_{\nu,0}}$ allows us to control the resolvent of the operator $K_{\nu,0}$. When doing so, a logarithmic factor appears, with optimality up to an exponent.   
\begin{lem}
For every $t\ge0$, one has
 $$\|e^{-tK_{\nu,0}}\|_{\mathcal{L}(L^2(\mathbb{R}^2))}=e^{-\operatorname{Argsh}\Big(S(t)\Big)}$$
where \begin{align*}
S(t)=\frac{\sh(\frac{tn_1}{2})}{n_1}=\frac{\sh(\frac{t\sqrt{4\nu+1}}{2})}{\sqrt{4\nu+1}}~.
\end{align*}

\end{lem}
\begin{proof}
Using (\ref{eq18}) and (\ref{eq19}), we directly compute that
\begin{align*}
\overline{(\kappa(t))^{-1}}\;\kappa(t):=\overline{e^{-itH_{K_{\nu,0}}}}e^{itH_{K_{\nu,0}}}=e^{itE}\Big(a+bI-cJ\Big)\Big(a+bI+cJ\Big)~,
\end{align*}
with $a=C(t),\;b=iS(t)\; \text{and} \;c=2izS(t)~.$

Note that $\Big(a+bI-cJ\Big)\Big(a+bI+cJ\Big)
=a^2-b^2+c^2+v$. Furthermore, $a^2+b^2+c^2=1$ and  $(a^2-b^2+c^2)^2+N(v)=1$. It follows that $N(v)=1-(a^2-b^2+c^2)^2=1-(1-2b^2)^2=4b^2(1-b^2).$

Denote $\sh(u)=\sqrt{-b^2}$, so $\sqrt{-N(v)}=2\sh(u)\ch(u)=\sh(2u).$

The eigenvalues of $\overline{(\kappa(t))^{-1}}\;\kappa(t)$ are given by
\begin{align*}&\frac{1}{\mu_1}=e^{t}(a^2-b^2+c^2+\sqrt{-N(v)})\\&\mu_1=e^{-t}(a^2-b^2+c^2-\sqrt{-N(v)})\\&\frac{1}{\mu_2}=e^{t}(a^2-b^2+c^2-\sqrt{-N(v)})\\&\mu_2=e^{-t}(a^2-b^2+c^2+\sqrt{-N(v)})~.\end{align*}
Therefore (see \cite{Vio2} Theorem 1.3),
 $$\|e^{-tK_{\nu,0}}\|_{\mathcal{L}(L^2(\mathbb{R}^2))}=(\mu_1\frac{1}{\mu_2})^{\frac{1}{4}}=e^{-\frac{1}{2}\operatorname{Argsh}(\sqrt{-N(v)})}=e^{-Argsh(\sqrt{-b^2})}~,$$
where $$-b^2=\Big(S(t)\Big)^2=\Big(\frac{\sh(\frac{tn_1}{2})}{n_1}\Big)^2~.$$
\end{proof}

\begin{prop}There exists some $c > 0$ such that, for all $\nu > c$,
\begin{align*}
\|K_{\nu,0}^{-1}\|_{\mathcal{L}(L^2(\mathbb{R}^2))}\le  c\frac{\log(\nu)}{\sqrt{\nu}}~.
\end{align*}
\end{prop}

\begin{proof} Observing that 
\begin{align*}
\|K_{\nu,0}^{-1}\|_{\mathcal{L}(L^2(\mathbb{R}^2))}=\|\displaystyle\int_0^{+\infty}e^{-tK_{\nu,0}}dt\|_{\mathcal{L}(L^2(\mathbb{R}^2))}\le \int_0^{+\infty}\|e^{-tK_{\nu,0}}\|_{\mathcal{L}(L^2(\mathbb{R}^2))}dt{~,}
\end{align*}
we aim to obtain an upper bound of the right-hand side.

Using the exact norm of the semigroup generated by $K_{\nu,0}$, we write   \begin{align*}
\int_0^{+\infty}\|e^{-tK_{\nu,0}}\|_{\mathcal{L}(L^2(\mathbb{R}^2))}dt&=\int_0^{+\infty}e^{-\operatorname{Argsh}\Big(\frac{\sh(\frac{tn_1}{2})}{n_1}\Big)}dt=\int_0^{+\infty}\frac{1}{\frac{\sh(\frac{tn_1}{2})}{n_1}+\sqrt{1+\Big(\frac{\sh(\frac{tn_1}{2})}{n_1}\Big)^2}}dt\\&=\int_0^{\log(\nu)}\frac{2du}{\sh(u)+\sqrt{n_1^2+\sh^2(u)}}+\int_{\log(\nu)}^{+\infty}\frac{2du}{\sh(u)+\sqrt{n_1^2+\sh^2(u)}}\\&\le 2\Big({\frac{\log(\nu)}{n_1}}+\int_{\log(\nu)}^{+\infty}e^{-u}du\Big)\\&\le 2\Big({\frac{\log(\nu)}{n_1}}+\frac{1}{\nu}\Big)\le c\;\frac{\log(\nu)}{\sqrt{\nu}}~.
\end{align*}
 This completes the proof.

\end{proof}

\subsection{Optimality with a logarithmic factor}
\begin{prop} One can find a function $u\in L^2(\mathbb{R}^2)$ such that
\begin{align*}
\|K_{\nu,0}u\|_{L^2(\mathbb{R}^2)}\le c\;\frac{\sqrt{\nu}}{\sqrt{\log(\nu)}}\|u\|_{L^2(\mathbb{R}^2)}
\end{align*}
where $c>0$ is a constant that does not depend on the parameter $\nu\gg 1$.
\end{prop}
\begin{proof} We recall here that \begin{align*}K_{\nu,0}=\frac{1}{2}(-\partial_p^2+p^2)+\sqrt{\nu}\Big(p\partial_q+q\partial_p\Big)=O_p+\sqrt{\nu}X_{0}~.
\end{align*}
For all $u\in D(K_{\nu,0})$,
\begin{align*}
\|K_{\nu,0}u\|^2_{L^2(\mathbb{R}^2)}\le 2\Big(\|O_pu\|^2_{L^2(\mathbb{R}^2)}+\nu \|X_0u\|^2_{L^2(\mathbb{R}^2)}\Big)~,
\end{align*}
then to prove the Proposition we will look for a function $u\in L^2(\mathbb{R}^2)$  such that $$\frac{\|O_pu\|^2_{L^2(\mathbb{R}^2)}+\nu \|X_0u\|^2_{L^2(\mathbb{R}^2)}}{\|u\|^2_{L^2(\mathbb{R}^2)}}\le c\;\frac{\nu}{\log(\nu)}~.$$
Consider the Gaussian
\begin{align*}
\varphi(q,p)=\frac{e^{-\frac{(q^2+p^2)}{2}}}{\sqrt{\pi}}
\end{align*}
and set
\begin{align*}
u(q,p)=\frac{1}{L}\int_{0}^Le^{sX_0}\varphi ds=\frac{1}{L}\int_{0}^L\varphi_s(q,p) ds
\end{align*}
where $\varphi_s(q,p)=e^{sX_0}\varphi(q,p)$ and $L>0$ is a constant to be specified at the end of the proof.

One has 
\begin{align*}\frac{d}{ds}\varphi_s=X_0(\varphi_s)=(p\partial_q+q\partial_p)\varphi_s{~.}\end{align*}

Let $\Big(q(t),p(t)\Big)$ be the solution of the following system: 
\begin{align*}
  \left\{
    \begin{array}{ll}
   \frac{d}{dt}q=p\\\\\  \frac{d}{dt}p=q\end{array}
\right.
\end{align*}
with $(q(0), p(0)) = (q_0, p_0)$. The solution is given by
\begin{align*}
 \left\{
    \begin{array}{ll}
    q(t)=\ch(t)q_0+\sh(t)p_0\\\\\  p(t)=\sh(t)q_0
    + \ch(t)p_0~.\end{array}
\right.
\end{align*}
The function $\varphi_s$ verifies
$$ \frac{d}{ds}\Big( \varphi_s(q(-s),p(-s))\Big)=\frac{\partial}{\partial s}\varphi_s-\frac{d}{ds}q(-s)\partial_q \varphi_s-\frac{d}{ds}p(-s)\partial_p \varphi_s=0~,$$
then
\begin{align*}\varphi_s(q,p)&=\varphi_0(q(s),p(s))=\varphi\Big(\ch(s)q+\sh(s)p,\sh(s)q+\ch(s)p\Big)\\&=\frac{1}{\sqrt{\pi}}e^{-\frac{\Big(\ch(s)q+\sh(s)p\Big)^2+\Big(\sh(s)q+\ch(s)p\Big)^2}{2}}{~.}\end{align*}
For all $p\in[0,+\infty]$, $\|\varphi_s\|_{L^p}=\|\varphi\|_{L^p}$. In particular,
$\|\varphi_s\|_{L^2}=\|\varphi\|_{L^2}=1~.$

Let's start by calculating $\|X_0u\|_{L^2(\mathbb{R}^2)}$:
\begin{align*}
X_0u=\frac{1}{L}\int_0^LX_0e^{sX_0}\varphi ds=\frac{1}{L}\int_0^L\frac{d}{ds}\varphi_s ds=\frac{1}{L}(\varphi_L-\varphi)~.
\end{align*}
As a result, \begin{align*}
\displaystyle\|X_0u\|^2_{L^2(\mathbb{R}^2)}&=\frac{1}{L^2}\|\varphi_L-\varphi\|^2_{L^2(\mathbb{R}^2)}=\frac{1}{L^2}\Big(\underbrace{\|\varphi_L\|^2_{L^2(\mathbb{R}^2)}}_{=1}+\underbrace{\|\varphi\|^2_{L^2(\mathbb{R}^2)}}_{=1}-2\int_{\mathbb{R}^2}\varphi_L\varphi\; dqdp\Big)\\&=\frac{2}{L^2}\Big(1-\int_{\mathbb{R}^2}\varphi_L\varphi\;dqdp\Big){~.}
\end{align*}
We directly compute that
\begin{align*}
\displaystyle\int_{\mathbb{R}^2}\varphi_L(q,p)\varphi(q,p) dqdp&=\frac{1}{\pi}\int_{\mathbb{R}^2}e^{-\frac{\Big(\ch(L)q+\sh(L)p\Big)^2+\Big(\sh(L)q+\ch(L)p\Big)^2}{2}}e^{-\frac{q^2+p^2}{2}}dqdp\\&=\frac{1}{\pi}\int_{\mathbb{R}^2}e^{-\frac{1}{2}\Big[2\ch^2(L)q^2+2\ch^2(L)p^2+4\sh(L)\ch(L)qp\Big]}dqdp\\&=\frac{1}{\pi}\int_{\mathbb{R}^2}e^{-\frac{1}{2}(q,p)A\;^t(q,p)}dqdp=\frac{1}{\pi}\sqrt{\frac{(2\pi)^2}{det(A)}}=\frac{1}{\ch(L)}
\end{align*}
where \begin{align*}A=\begin{pmatrix}
2\ch^2(L)&2\ch(L)\sh(L)\\\\2\ch(L)\sh(L)&2\ch^2(L)
\end{pmatrix}~.\end{align*}
Then
\begin{align}
\|X_0u\|^2_{L^2(\mathbb{R}^2)}=\frac{2}{L^2}\Big(1-\frac{1}{\ch(L)}\Big)~.\label{eq8}
\end{align}
Now, let's find a lower bound for $\|u\|^2_{L^2(\mathbb{R}^2)}$:
\begin{align*}
\|u\|^2_{L^2(\mathbb{R}^2)}&=\frac{1}{L^2}\int_0^{L}\int_0^{L}\mathbb{R}\text{e}\langle \varphi_{s_1},\varphi_{s_2}\rangle_{L^2(\mathbb{R}^2)}ds_1ds_2\\&=\frac{2}{L^2}\int_0^{L}\Big[\int_{s_1}^{L}\mathbb{R}\text{e}\langle \varphi_{s_1},\varphi_{s_2}\rangle_{L^2(\mathbb{R}^2)}ds_2\Big]ds_1\\&\underset{s_2=s_1+s}{=}\frac{2}{L^2}\int_0^{L}\Big[\int_{0}^{L-s_1}\mathbb{R}\text{e}\langle \varphi_{s_1},\varphi_{s_1+s}\rangle_{L^2(\mathbb{R}^2)}ds\Big]ds_1{~.}
\end{align*}
But 
\begin{align*}
\mathbb{R}\text{e}\langle \varphi_{s_1+s},\varphi_{s_1}\rangle_{L^2(\mathbb{R}^2)}&=\langle e^{s_1X_0}\varphi,e^{(s_1+s)X_0}\varphi\rangle_{L^2(\mathbb{R}^2)}\\&=\langle e^{s_1X_0}\varphi,e^{sX_0}\varphi\rangle_{L^2(\mathbb{R}^2)}\\&=\int_{\mathbb{R}^2}\varphi_s(q,p)\varphi(q,p) dqdp=\frac{1}{\ch(s)}{~.}
\end{align*}
For $L>2$ we obtain
\begin{align}
\displaystyle\|u\|^2_{L^2(\mathbb{R}^2)}&=\frac{2}{L^2}\int_0^L\Big[\int_0^{L-s_1}\frac{1}{\ch(s)}ds\Big]ds_1\nonumber\\&\ge\frac{2}{L^2}\int_0^{\frac{L}{2}}\Big[\int_0^{\frac{L}{2}}\frac{1}{\ch(s)}ds\Big]ds_1\ge\frac{2}{L^2}\int_0^{\frac{L}{2}}\Big[\int_0^{1}\frac{1}{\ch(s)}ds\Big]ds_1\nonumber\\&\ge 
\label{eq9} \frac{c}{L}{~.}\end{align}
The final step is the upper bound of $\|O_pu\|^2_{L^2(\mathbb{R}^2)}$:
\begin{align*}
\|O_pu\|^2_{L^2(\mathbb{R}^2)}&=\|O_p\Big(\frac{1}{L}\int_0^L\varphi_s(q,p)ds\Big)\|^2_{L^2(\mathbb{R}^2)}\le \frac{1}{L^2}\int_0^L\|O_p\varphi_s\|^2_{L^2(\mathbb{R}^2)} ds{~.}
\end{align*}

With $O_p = \frac{1}{2}(D_p^2 + p^2)$, we want to compute
\[
	\|O_p \varphi_s\|_{L^2(\Bbb{R}^2)} = \|e^{-sX_0}O_p e^{sX_0}\varphi_0\|_{L^2(\Bbb{R}^2)}
\]
(because $e^{-sX_0}$ is unitary and $\varphi_s = e^{sX_0}\varphi_0$).

For any $u \in L^2(\Bbb{R}^2)$, $e^{sX_0}u(q,p) = u(e^{sM}(q,p))$ where
\[
	e^{sM} = \left(\begin{array}{cc} \ch s & \sh s \\ \sh s & \ch s\end{array}\right).
\]
Egorov's theorem gives that, for any symbol $a(q,p,\xi_q, \xi_p)$, 
\[
	e^{-sX_0}a^w(q,p,D_q,D_p)e^{sX_0} = a^w(e^{-sM}(q,p), e^{sM}(D_q,D_p)).
\]
In particular, writing $O_q = \frac{1}{2}(D_q^2 + q^2)$ as well,
\[
	\begin{aligned}
	e^{-sX_0}(p^2 + D_p^2)e^{sX_0} &= (-\sh(s)q + \ch(s)p)^2 + (\sh(s)D_q + \ch(s)D_p)^2
	\\ &= \sh^2(s)q^2 - 2\ch(s)\sh(s)qp+\ch^2(s)p^2 
	\\ &\qquad + \sh^2(s)D_q^2 + 2\ch(s)\sh(s)D_qD_p + \ch^2(s)D_p^2
	\\ &= 2\ch^2(s)O_q + 2\sh^2(s)O_p + 2\ch(s)\sh(s)(D_q D_p - qp)~.
	\end{aligned}
\]

We have chosen $\varphi_0$ an eigenfunction of both $O_p$ and $O_q$ with eigenvalue $\frac{1}{2}$, and $D_qD_p\varphi_0 = -qp\varphi_0$. Therefore
\[
	e^{-sX_0}O_p e^{sX_0} \varphi_0 = \left(\frac{1}{2}(\ch^2(s) + \sh^2(s)) - 2\ch(s)\sh(s) qp \right)\varphi_0~.
\]
This can be interpreted as the sum of products of the first two orthornormal Hermite functions: if
\[
	h_0(x) = \pi^{-1/4}e^{-x^2/2}, \quad h_1(x) = \sqrt{2}xh_0(x)~,
\]
then $\varphi_0(q,p) = h_0(q)h_0(p)$ and
\[
	e^{-sX_0}O_p e^{sX_0} \varphi_0 = \frac{1}{2}(\ch^2(s)+\sh^2(s))h_0(q)h_0(p) - \ch(s)\sh(s)h_1(q)h_1(p).
\]
This type of tensor product forms an orthonormal family, so by the Pythagorean relation the square of the norm can be computed as the sum of squares of the coefficients:
\[
	\|O_p\varphi_s\|_{L^2(\Bbb{R}^2)}^2 = \|e^{-sX_0}O_p e^{sX_0}\varphi_0\|_{L^2(\Bbb{R}^2)}^2 = \frac{1}{4}(\ch^2(s) + \sh^2(s))^2 + \ch^2(s)\sh^2(s) = \frac{1}{4}\ch(4s)~.
\]
Thus we deduce that 
 \begin{align}
 \|O_pu\|^2_{L^2}&\le\frac{1}{L^2}\int_0^L e^{4s}ds=\frac{1}{4L^2}(e^{4L}-1)\nonumber\\&\le \frac{1}{L^2}e^{4L}{~.}\label{eq10}\end{align}
The estimates in (\ref{eq8})  and (\ref{eq9}) taken with (\ref{eq10}), allow us to establish that
\begin{align*}
\frac{\|K_{\nu,0}u\|_{L^2}^2}{\|u\|_{L^2}^2}\le\frac{\|O_pu\|_{L^2}^2+\nu \|X_0u\|_{L^2}^2}{\|u\|_{L^2}^2}\le c\;\frac{e^{4L}+\nu\Big(1-\frac{1}{\ch(L)}\Big)}{L}{~.}
\end{align*}
Now letting $L=\frac{\log(\nu)}{4}~,$ we get  the desired inequality
\begin{align*}
\|K_{\nu,0}u\|_{L^2}^2\le c\;\frac{\nu}{\log(\nu)}\|u\|_{L^2}^2{~.}
\end{align*}
 \end{proof}

\section{Degenerate one-dimensional case}
\label{sec:deg}
 \begin{lem} Let ${\lambda_1\in \mathbb{R}}$ be parameter. Consider the operator  \begin{align*} K_1={p.\partial_q-\lambda_1\partial_p}+\frac{1}{2}(-\partial^2_p+p^2-1)
 \end{align*}   with domain $D(K_1)=\left\{u\in L^2(\mathbb{R}^2),\;\;\; K_1u\in  L^2(\mathbb{R}^2)\right\}$.
 There exists a constant $c>0$ such that \begin{align*}
\|(D_q^2+{\lambda_1^2})e^{-t(K_1+1)}\|_{\mathcal{L}(L^2(\mathbb{
 R}^2))}\le\frac{c}{t^3}
 \end{align*} 
holds  for all $t>0$.
 \end{lem} 
 \begin{proof}
 
For each $\xi_q$ fixed, there is a metaplectic operator on $L^2(\mathbb{R}_p)$ which, via conjugation, takes $i p.\xi_q - i\lambda_1 D_p$ to $ip\sqrt{\xi_q^2 + \lambda_1^2}$ while leaving $O_p$ invariant. Taking the direct integral of this rotation (whose angle depends on $\xi_q$) gives a unitary equivalence between the operator $K_1$ and 
\[
	\widehat{K_1} =\frac{1}{2}\Big(2ip\sqrt{D_q^2+{\lambda_1^2}}+(-\partial^2_p+p^2-1)\Big).
\]
We also note that $\sqrt{D_q^2 + \lambda_1^2}$ is left invariant by the rotation in the variables $(p, \xi_p)$.

It is shown in \cite{Vio2} that 
\begin{align*}
  \|e^{-i(t_1+it_2)P_b}\|_{\mathcal{L}(L^2(\mathbb{R}))}=\exp\Big(\frac{\cos(t_1)-\ch(t_2)}{\sh(t_2)}b^2\Big)
\end{align*}
for all $t_1\in\mathbb{R}$ and all $t_2<0$~, where $P_b=\frac{1}{2}\Big(D^2_x+x^2-1+2ibx{-}b^2\Big),\;\;\;b\in\mathbb{R}$.
Applying this result with $t_1=0$, $t_2=-t<0$ and $b=b(\xi_q) = \sqrt{\xi_q^2+{\lambda_1^2}}~$, we obtain
\begin{align*}
	\|\sqrt{D_q^2 + \lambda_1^2}e^{-t\widehat{K_1}}\|_{\mathcal{L}(L^2(\mathbb{R}^2))} & \leq \sup_{\xi_q \in \mathbb{R}}\|b^2e^{-t(P_b + \frac{b^2}{2})}\|_{\mathcal{L}(L^2(\mathbb{R}^2))} = \sup_{\xi_q \in \mathbb{R}} b^2e^{-\frac{t}{2}b^2}e^{(\frac{\ch(t)-1}{\sh(t)})b^2}{~.}
\end{align*}
(We remark that this inequality can be strengthened to an equality by taking the tensor product of explicit optimisers for the norm of $e^{-tP_b}$ with functions in $q$ localized in phase space near the optimising $\xi_q$.)

For all $t\in [0,1]$, denote $f_b(t)=b^2e^{(\frac{\ch(t)-1}{\sh(t)}-\frac{t}{2})b^2}$, and $u(t)=\frac{\ch(t)-1}{\sh(t)}-\frac{t}{2}=\tnh(\frac{t}{2})-\frac{t}{2}<0$.

Since $\max\limits_{x\in\mathbb{R}}\;xe^{-ax}=\frac{e^{-a}}{a}$ when $a>0$\,, we get
  \begin{align*}
  b^2t^3\exp\Big(u(t)b^2\Big)\le \frac{-t^3e^{u(t)}}{u(t)}=:F(t)\,.
  \end{align*}
 The expansion
$u(t)=\tnh(\frac{t}{2})-\frac{t}{2}=\frac{-t^3}{24}+\mathcal{O}(t^4)$
yields
 $\lim_{t\to 0} F(t)=24$ and the function $F$ is bounded on
 the interval $[0,1]$\,. Replacing $b^{2}$ with $D_{q}^{2}+\lambda_{1}^{2}$\,, we conclude that, for $t\in[0,1]$\,,
  \begin{align*}
\|(D_q^2+{\lambda_1^2})e^{-t(K_1+1)}\|_{\mathcal{L}(L^2(\mathbb{
 R}^2))}\le\frac{c}{t^3}~.
  \end{align*}
 For all $t\ge1$, just write with $t_{0}=\frac{1}{2}$\,,
\begin{align*}
 \|(D_q^2+{\lambda_1^2})e^{-t(K_1+1)}\|_{\mathcal{L}(L^2(\mathbb{
 R}^2))}\le\underbrace{\|(D_q^2+{\lambda_1^2})e^{-t_0K_1}\|_{\mathcal{L}(L^2(\mathbb{
 R}^2))}}_{\le \frac{c}{t_0}}\underbrace{\|e^{-(t-t_0)K_1}\|_{\mathcal{L}(L^2(\mathbb{
 R}^2))}}_{\le 1}e^{-\frac{t}{2}}\le\frac{c}{t^3}~.
 \end{align*} 
\end{proof}

\begin{appendices}
 \section{Biquaternions}

We define a biquaternion $W$ as follows:
$$W=a+b\textbf{i}+c\textbf{j}+d\textbf{k}$$ where $a,b,c,d$ are complex numbers and  $\textbf{i}, \textbf{j}, \textbf{k}$ multiply according to the rules
\begin{align*}
&\textbf{i}^2=\textbf{j}^2=\textbf{k}^2=\textbf{i}\textbf{j}\textbf{k}=-1\\&
\textbf{i}\textbf{j}= -\textbf{j}\textbf{i} = \textbf{k}\\&\textbf{j}\textbf{k}= -\textbf{k}\textbf{j} = \textbf{i}\\&\textbf{k}\textbf{i} = -\textbf{i}\textbf{k} = \textbf{j}~.\end{align*}
For convenience we use a vector notation for biquaternions as follows:

$$W=a+v~,\;\;\;
v=b\textbf{i}+c\textbf{j}+d\textbf{k}~.$$
The conjugate of a biquaternion $W$ is given by

$$\text{conj}(W) =a-b\textbf{i}-c\textbf{j}-d\textbf{k}~.$$

The biquaternion ring $B_{Q }$ is isomorphic to  the matrix ring $\mathbb{M}_2(\mathbb{C})$. This can be seen via the following map:
\begin{align*}
&f:B_{Q }\to\mathbb{M}_2(\mathbb{C})\\&a+b\textbf{i}+c\textbf{j}+d\textbf{k}\mapsto M=\begin{pmatrix}
         a+b\textbf{i}&c+d\textbf{i}\\-c+d\textbf{i}&a-b\textbf{i}
\end{pmatrix}{~.}
\end{align*}

The "norm" $N(W)$ of a biquaternion  $W$ is
\begin{align*}N(W)=\mathrm{conj}(W)W=\mathrm{ det}(M)=a^2+b^2+c^2+d^2~.\end{align*}
Note that the norm is homogeneous of degree 2 and may take complex values. In particular, a biquaternion $W$ is invertible if and only if $N(W)\neq 0$. In this case its inverse is given by

$$\text{inv}(W)=\frac{\mathrm{conj}(W)}{N(W)}~.$$
\noindent\textbf{Exponential and spectrum. }

Let $a+b\textbf{i}+c\textbf{j}+d\textbf{k}=a+v$ be a biquarternion such that $N(v)\neq 0$. In this case $ \widehat{v}=\frac{v}{\sqrt{N(v)}}$ verifies $ \widehat{v}^2=-1$. 

Hence write 

\begin{align}
e^{a+v}&=e^{a}e^{v}=e^{a}e^{\widehat{v}\sqrt{N(v)}}\nonumber
\\&=e^{a}\sum\limits_{k=0}^{+\infty}\frac{(\widehat{v}\sqrt{N(v)})^k}{k!}\nonumber
\\&=e^{a}\Big(\sum\limits_{k=0}^{+\infty}\frac{(-1)^k(\sqrt{N(v)})^{2k}}{(2k)!}+\sum\limits_{k=0}^{+\infty}\frac{(-1)^k(\sqrt{N(v)})^{2k+1}}{(2k+1)!}\widehat{v}\Big)\nonumber
\\&=e^{a}\Big(\cos(\sqrt{N(v)})+\frac{\sin(\sqrt{N(v)})}{\sqrt{N(v))}}v\Big){~.}\label{formula1}
\end{align}
The above computation do not depend on the choice of $\sqrt{N(v)}$
because $\cos$ is even and $\sin$ is odd.\\
Finally  the set of $\lambda\in \mathbb{C}$ such that $(a+v-\lambda)$ is non-invertible can be explicitly determined:

$(a+v-\lambda)$ is non-invertible if and only if $ 0=N(a+v-\lambda)=(a-\lambda)^2+N(v)$ if and only if $\lambda\in\left\{a\pm\sqrt{-N(v)}\right\}$.

\end{appendices}

\end{document}